\documentclass[11pt]{amsart}
\usepackage{amsmath,amssymb,amsfonts,latexsym}
\usepackage{mathrsfs}%pour les caracteres en calligraphie
\usepackage[dvips]{geometry}

\usepackage{epsf}
\usepackage{epsfig}
\usepackage{xcolor}

\usepackage{tikz,tikz-cd}

\usetikzlibrary{matrix}

\usetikzlibrary{fit}
\usepackage{animate}
\usepackage{float}

\usetikzlibrary{matrix,backgrounds}
\usetikzlibrary{arrows,matrix,positioning}
\usepackage{pstricks}

\usetikzlibrary{graphs,graphs.standard}
\usepackage{changepage}
\usepackage{tikz-cd}

\usepackage{array}

\newtheorem*{lemma}{Lemma}
\newtheorem*{prop}{Proposition}
\newtheorem*{thm}{Theorem}
\newtheorem*{cor}{Corollary}

\newcommand{\twoheaddownarrow}{\overset{\sim}{\twoheaddownarrow}}

\newcommand{\gr}{\operatorname{gr}}
\newcommand{\nc}{\newcommand}
\nc{\Ker}{\operatorname{Ker}} \nc{\rker}{\operatorname{rKer}}
\nc{\im}{\operatorname{Im}}
\nc{\stab}{\operatorname {Stab}}
\nc{\ann}{\operatorname {Ann}}
\nc{\Id}{\operatorname {Id}}
\nc{\Prim}{\operatorname {Prim}}
\nc{\Real}{\operatorname {Re}}
\nc{\Ext}{\operatorname {Ext}}
\nc{\rad}{\operatorname {rad}}
\nc{\rk}{\operatorname {rank}}
\nc{\Aut}{\operatorname {Aut}}
\nc{\supp}{\operatorname {supp}}
\nc{\codim}{\operatorname {codim}}
\makeatletter
\newcommand*{\encircled}[1]{\relax\ifmmode\mathpalette\@encircled@math{#1}\else\@encircled{#1}\fi}
\newcommand*{\@encircled@math}[2]{\@encircled{$\m@th#1#2$}}
\newcommand*{\@encircled}[1]{%
 \tikz[baseline,anchor=base]{\node[draw,circle,outer sep=0pt,inner sep=.2ex] {#1};}}

\begin{document}

\title [Canonical Component]{The Canonical Component of the nilfibre for Parabolic adjoint action in type $A$}
\author [Yasmine Fittouhi and Anthony Joseph]{Yasmine Fittouhi and Anthony Joseph}
\date{\today}
\maketitle
\vspace{-.9cm}\begin{center}

Department of Mathematics\\
The University of Haifa\\
Haifa, 3498838, Israel\\
fittouhiyasmine@gmail.com
\end{center}\

\

\

\vspace{-.9cm}\begin{center}
Donald Frey Professional Chair\\
Department of Mathematics\\
The Weizmann Institute of Science\\
Rehovot, 7610001, Israel\\
anthony.joseph@weizmann.ac.il
\end{center}\

\

\date{\today}
\maketitle



 \

\textbf{Abstract}.

\

This work is a continuation of [Y. Fittouhi and A. Joseph, Parabolic adjoint action, Weierstrass Sections and components of the nilfibre in type $A$].

Let $P$ be a parabolic subgroup of an irreducible simple algebraic group $G$. Let $P'$ be the derived group of $P$ and let $\mathfrak m$ be the Lie algebra of the nilradical of $P$.

 A theorem of Richardson implies that the subalgebra $\mathbb C[\mathfrak m]^{P'}$, spanned by the $P$ semi-invariants in $\mathbb C[\mathfrak m]$, is polynomial.

A linear subvariety $e+V$ of $\mathfrak m$ is is called a Weierstrass section for the action of $P'$ on $\mathfrak m$, if the restriction map induces an isomorphism of $\mathbb C[\mathfrak m]^{P'}$ onto $\mathbb C[e+V]$. Thus a Weierstrass section can exist only if the latter is polynomial, but even when this holds its existence is far from assured.

Let $\mathscr N$ be zero locus of the augmentation $\mathbb C[\mathfrak m]^{P'}_+$. It is called the nilfibre relative to this action.

Suppose $G=SL(n,\mathbb C)$, and $P$ a parabolic subgroup.

In [Y. Fittouhi and A. Joseph, loc. cit.], the existence of a Weierstrass section $e+V$ in $\mathfrak m$ was established by a general combinatorial construction. Notably $e \in \mathscr N$ and is a sum of root vectors with linearly independent roots. The Weierstraass section $e+V$ looks very different for different choices of parabolics but nevertheless has a uniform construction and exists in all cases. It is called the ``canonical Weierstrass section''.

Through [Y. Fittouhi and A. Joseph, loc. cit. Prop. 6.9.2, Cor. 6.9.8], there is always a ``canonical'' component $\mathscr N^e$ of $\mathscr N$ containing $e$.

%Ideally $\mathscr N^e$ is the closure $\overline {P.e}$ of $P.e$ and so $\mathscr N^e$ is determined by $e$. Again ideally the Weierstrass section $e+V$ is also determined by $e$.

%In order to investigate this more fully, a more explicit description of the canonical Weierstrass section is given. In particular a key proposition announced in [Y. Fittouhi and A. Joseph, loc. cit., Prop. 6.10.4] is established.

%In fact $P.e$ may fail to be dense in $\mathscr N^e$ [Y. Fittouhi and A. Joseph, loc. cit., Lemma 6.10.7].

%It was announced in [Y. Fittouhi and A. Joseph, loc. cit., Prop. 6.10.4] that one may canonically augment $e$ to an element $e_{VS}$ by adjoining root vectors. Then the linear span $E_{VS}$ of these root vectors lies in $\mathscr N^e$ and moreover $\overline {P.E_{VS}}=\mathscr N^e$.

 It was announced in [Y. Fittouhi and A. Joseph, loc. cit., Prop. 6.10.4] that one may augment $e$ to an element $e_{VS}$ by adjoining root vectors. Then the linear span $E_{VS}$ of these root vectors lies in $\mathscr N^e$ and its closure is just $\mathscr N^e$.

% Thus even when $P.e$ fails to be dense, nevertheless $e$ determines $\mathscr N^e$.

Yet this same result shows that $\mathscr N^e$ need \textit{not} admit a dense $P$ orbit [Y. Fittouhi and A. Joseph, loc. cit., Lemma 6.10.7].

For the above [Y. Fittouhi and A. Joseph, loc. cit., Theorem 6.10.3] was needed. However thhis theorem was only verified in the special case needed to obtain the example showing that $\mathscr N^e$ may fail to admit a dense $P$ orbit. Here a general proof is given (Theorem \ref {4.4.5}).

Finally a map from compositions to the set of distinct non-negative integers is defined. Its image is shown to determine the canonical Weierstrass section.

One may anticipate that the remaining components of $\mathscr N$ can be similarly described. However this is a long story and will be postponed for a subsequent paper.

These results should form a template for general type.

\section{Introduction}\label{1}

The base field is the field of complex numbers $\mathbb C$. For every positive integer $n$, we set $[1,n]=\{1,2,,\ldots\,n\}$.

\subsection{}\label{1.1}

Let $G$ be a simple, connected algebraic group over $\mathbb C$ and $\mathfrak g$ its Lie algebra. Let $\mathfrak h$ be a Cartan subalgebra of $\mathfrak g$. Let $H$ be the unique connected closed subgroup of $G$ with Lie algebra $\mathfrak h$.

 Let $\Delta \subset \mathfrak h^*$ be the set of roots for the pair $(\mathfrak g, \mathfrak h)$. Let $\Delta^+$ be a choice of positive roots and $\pi$ the resulting set of simple roots. Let $\mathfrak n$ (resp. $\mathfrak n^-$) be the subalgebra of $\mathfrak g$ spanned by the positive (resp. negative) root vectors. Then $\mathfrak b:=\mathfrak h+\mathfrak n$ is a Borel subalgebra of $\mathfrak g$. It will remain fixed throughout.

The Weyl group of $G$ is the subgroup $W$ of $\Aut \mathfrak h$ generated by the reflections $s_\alpha: \alpha \in \Delta$, where $s_\alpha h =h -\alpha(h) \alpha^\vee$, for all $h \in \mathfrak h$, with $\alpha^\vee$ the coroot to $\alpha$.

\subsection{}\label{1.2}

A (standard) parabolic subalgebra of $\mathfrak g$ is one containing $\mathfrak b$.

For any subset $\pi' \subset \pi$, let $\mathfrak r_{\pi'}$ be the reductive Lie subalgebra of $\mathfrak g$ generated over $\mathfrak h$ by the root vectors $x_\alpha,x_{-\alpha};\alpha \in \pi'$. It is the Levi factor of a unique standard parabolic subalgebra $\mathfrak p_{\pi'}$. It is complemented by the nilradical $\mathfrak m_{\pi'}$ of $\mathfrak p_{\pi'}$ and indeed $\mathfrak m_{\pi'}$ is the span of the root subspaces $x_\beta: \beta \in \Delta^+ \setminus \mathbb N\pi' \cap \Delta^+$. All standard parabolics so obtain.

In what follows the $\pi'$ subscript will often be omitted.

Let $\mathfrak p'$ denote the derived algebra of $\mathfrak p$. Let $P$ (resp. $P'$) denote the unique closed subgroup of $G$ with Lie algebra $\mathfrak p$ (resp. $\mathfrak p'$).

\subsection{}\label{1.3}

As a consequence of a theorem of Richardson, $\mathbb C[\mathfrak m]^{P'}$ is a polynomial algebra - \cite [2.2.3]{FJ1}.

The nilfibre $\mathscr N$ for the action of $P$ on $\mathfrak m$ is defined to be the zero locus of the augmentation $\mathbb C[\mathfrak m]^{P'}_+$.

Given $e \in \mathfrak m$ and $V$ a linear subspace of $\mathfrak m$, one calls $e+V$ a linear subvariety of $\mathfrak m$.

A Weierstrass section for the action of $P$ on $\mathfrak m$ is a linear subvariety $e+V \subset \mathfrak m$ such that restriction $\varphi$ of functions induces an isomorphism of $\mathbb C[\mathfrak m]^{P'}$ onto $\mathbb C[e+V]$.

A Weierstrass section realizes rather explicitly the polynomiality of $\mathbb C[\mathfrak m]^{P'}$. It also has the significant geometric interpretation that every $P'$ orbit meeting $e+V$, meets the latter at just one point, so providing a canonical representative in that orbit. Polynomiality is not enough to ensure the existence of a Weierstrass section \cite [1.2.7]{FJ1}.

The main result in \cite [5.4.12] {FJ2} was the existence of a Weierstrass section $e+V$ for the action of a parabolic on its nilradical in type $A$. The proof was rather combinatorial but nevertheless essentially canonical. We call $e+V$ the canonical Weierstrass section.

%In this $e$ is constructed as a sum of root vectors with linearand $V$ as a sum of root subspaces. Let

A main point is that $e$ belongs to a canonical component $\mathscr N^e$ of $\mathscr N$ which can be determined. We described $\mathscr N^e$ in some detail in \cite [Cor. 6.9.8]{FJ2}, which we reproduce here in \ref {2.7}. Here we start by extending this description proving in particular \cite [Prop. 6.10.4]{FJ2}, given here as Prop. \ref {4.5.3}, which we previously only verified in some needed special cases.

\subsection{}\label{1.4}

In order to place our results in a form that may more easily generalize to all $G$ simple, let us describe how the canonical Weierstrass section $e+V$, noted in \ref {1.3}, determines the canonical component $\mathscr N^e$. Then one can ask whether from a component of $\mathscr N$ one may construct a Weierstrass section.

Call $e \in \mathscr N$ regular if it generates a dense $P$ orbit in a component of $\mathscr N$ thereby determining this component.

Let $\mathscr N_{reg}$ denote the subset of all regular elements in $\mathscr N$.

Of course a Weierstrass section is by no means unique. However when $e \in \mathscr N_{reg}$, then any other Weierstrass section with this property is determined up to a choice of component and up to conjugation.

Thus an ``abstract'' description of all Weierstrass sections should start by describing the $P$ orbits in $\mathscr N_{reg}$. Then given $e \in \mathscr N_{reg}$ one should choose $V$ as a vector space complement to $\mathfrak p.e$ in $\mathfrak m$. This is how the ``Kostant Section'' is obtained when $\mathfrak p =\mathfrak g$. In that case
$\mathscr N_{reg}$ is irreducible.

A first difficulty is that in the construction of a Weierstrass section $e+V$ of \cite [5.4.12]{FJ2} the resulting $e$ may be ``too small'' to be regular.

This is partly overcome by augmenting $e$ to an element $e_{VS}$ by adjoining \cite [6.10.4]{FJ2} canonically chosen co-ordinate vectors to $e$, which we call $VS$ elements following the work of Victoria Sevostyanova \cite {Se}.

A second (serious) difficulty is that $\mathscr N_{reg}$ may be empty. In \cite [Lemma 6.10.7]{FJ2} it was shown that a component of $\mathscr N$ may contain no regular orbit. %(In this example it was expected that $\mathscr N$ be irreducible.)

Yet in our construction, $e_{VS}$ is a sum of weight vectors. Call $e_{VS}$ \textbf{quasi-regular} if the corresponding direct sum $E_{VS}$ of weight subspaces has the property that $P.E_{VS}$ is dense in some necessarily unique component $\mathscr N^e$ of $\mathscr N$. Then $\mathscr N^e$ is determined by $e_{VS}$ and so by $e$. It is called the canonical component.

In Section \ref {3}, we give a proof of the quasi-regularity $e_{VS}$. \textit{It is by no means easy}.

The main point in the distinction between regular and quasi-regular elements is that $H.e_{VS}$ need not be dense in $E_{VS}$ - \cite [6.10.7]{FJ2}.

Our aim is to show that every component of $\mathscr N$ admits a quasi-regular element. However it is by no means obvious that the existence of a quasi-regular element should lead to a Weierstrass section, nor that such a construction would extend to all types. For the moment it remains a signpost along the road to the solution of an immensely difficult problem.

\subsection{}\label{1.5}

An orbital variety closure is the closure of a set of the form $B.(\mathfrak n \cap w(\mathfrak n)): w \in W$. It is called a hypersurface orbital variety if it is also a hypersurface in some nilradical of a parabolic subalgebra. As noted in \cite [Lemma 2.3.4]{FJ1} the generators of $\mathbb C[\mathfrak m]^{P'}$ are put into bijection with the hypersurface orbital varieties in $\mathfrak m$, by taking zero loci. This holds for all $\mathfrak g$ simple, but only in type $A$ can one easily determine the hypersurface orbital varieties \cite [Prop. 2.17]{JM}. They have been described in classical types by Perelman \cite{P} but the details are far from transparent and so difficult to translate to a general case-free form. Here as briefly discussed in \cite [3.4]{FJ1}, the real difficulty lies in the subtle nature of the Springer construction of simple Weyl group modules associated to non-trivial representations of the component group of a given nilpotent orbit and as a consequence we do not know how to determine such modules arising from hypersurface orbital varieties (in type $E$) - see \cite [Remark 3.4.9]{FJ1}.

\subsection{}\label{1.6}

More generally let $\mathfrak u$ be a subalgebra of $\mathfrak n$. The $B$ saturation set defined by $\mathfrak u$ is just the closure $\overline{B.\mathfrak u}$ of $B.\mathfrak u$. Remarkably the canonical component $\mathscr C$ takes this form \cite [Cor. 6.9.8]{FJ2} and moreover \cite [6.9.7]{FJ2} has dimension $d:=\dim \mathfrak m -g$. \textit{The proof is far from easy}.

Our eventual aim is to show that every irreducible component of the nilfibre is a $B$ saturation set defined by a subalgebra of $\mathfrak n$ and to determine this subalgebra (which need not be unique). In general this $B$ saturation set is not an orbital variety closure. As noted in \cite [6.9.9]{FJ2}, this holds if and only if $\dim G.\mathfrak u= 2 \dim B. \mathfrak u$. Examples $3,4$ of \cite [6.10.9]{FJ2} show that this can hold and this can fail.

% \subsection{}\label{1.7}
%
% In \cite [6.10.8]{FJ2} we constructed some components of the nilfibre in type $A$. Here we show that all the good properties listed above of the canonical component $\mathscr C$ carry over to these components. We believe this constructs all components; but this is far from obvious.

\section {Type $A$ - The Zero Locus of the Benlolo-Sanderson Invariant.}\label {2}

\subsection{}\label{2.1}

From now on we assume that $G$ is simple of type $A$, that is isomorphic to $SL(n,\mathbb C)$ for some integer $n>1$. In this case the Weyl group $W$ is just the symmetric group $S_n$ on $n$ letters. Again the component group pf any orbit is trivial and so Springer theory recovers the classical fact that there is a bijection between nilpotent orbits and simple Weyl group modules (in type $A$).

Let $\textbf{M}_n$, or simply $\textbf{M}$, denote the set of $n \times n$ matrices over $\mathbb C$ and in this we write we write $x_{i,j}:i,j \in [1,n]$, for the standard matrix units. It is sometimes called the $(i,j)^{th}$ co-ordinate function, or co-ordinate vector, which assigns to $\textbf{M}$ its value at its $(i,j)^{th}$ entry.

Then the Lie algebra $\mathfrak {sl}(n,\mathbb C)$ of $SL(n,\mathbb C)$ is just the subspace $\textbf{M}_n$ of traceless matrices over $\mathbb C$ given a Lie bracket through the commutator.

Let $\alpha_i:i \in [1,n-1]$ be the simple roots in the Bourbaki notation \cite {Bo} and set $\alpha_{i,j}:=\alpha_i+\cdots+\alpha_{j-1}$. If $i<j$, then $x_{i,j}$ is the root vector $x_{\alpha_{i,j}}\in \mathfrak n$.

\subsection{}\label{2.2}

Let $\pi'$ be a subset of $\pi$. Let $W_{\pi'}$ be the subgroup of $W$ generated by the simple reflections $s_\alpha:\alpha \in \pi'$ and $w_{\pi'}$ it unique longest element.

Recall that the Levi factor $\mathfrak r_{\pi'}$ of $\mathfrak p_{\pi'}$ is given by a set of $c_i \times c_i$ blocks $\{B_i\}_{i=1}^k$ on the diagonal of $\textbf{M}_n$, where the $c_i-1$ are the cardinalities of the connected components of $\pi'$. One has $\sum_{i=1}^kc_i= n$. It is complemented in $\mathfrak p_{\pi'}$ by the nilradical $\mathfrak m_{\pi'}$ of $\mathfrak p_{\pi'}$, lying in $\textbf{M}_n$ above this set of blocks.

Often $\pi'$ will be viewed as being fixed and the $\pi'$ subscript omitted.

 Let $\textbf{C}_i$ denote the rectangular block in $\mathfrak m$ lying \textit{strictly above} $B_i$. We call it the $i^{th}$ column block. Its width is $c_i$ and its height $\sum_{j < i}c_j$. In this presentation, $\mathfrak m=\oplus_{i=2}^k \textbf{C}_i$.

 \subsection{}\label{2.3}

We regard $\mathfrak p$ as being defined by the composition $(c_1,c_2,\ldots,c_k)$, hence by a diagram $\mathscr D_\mathfrak m$ formed by a set of columns $\{C_i\}_{i=1}^k$ labelled from left to right with ht $C_i=c_i$.

Let $\{R_j\}_{j=1}^\ell$ be the rows of $\mathscr D_\mathfrak m$ labelled with increasing positive integers from top to bottom. For all $i \in \mathbb N^+$, let $R^i$ denote the union of the rows $\{R_j\}_{j=1}^i$.

All columns are deemed to start from the top so meeting row $R_1$.

Let $\mathscr T_\mathfrak m$, be the tableau obtained from $\mathscr D_\mathfrak m$ in which the integers $1,2,\ldots,n$ are inserted sequentially, first down the columns and then on going from left to right. In this we often drop the $\mathfrak m$ subscript, since $\mathfrak m$ will usually be fixed and to avoid cumbersome notation.

This presentation has the following property. Let $b,b'$ be boxes of $\mathscr T$ with $b'$ strictly to the right of $b$. Let $i$ (resp. $j$) be the entry of $b$ (resp. $b'$). Then $x_{i,j} \in \mathfrak m$ and these elements form a basis of $\mathfrak m$. Let $\ell_{b,b'}$ be the line joining $b,b'$. In the above labelling it is sometimes written as $\ell_{i,j}$ and said to join $i,j$. It defines the co-ordinate vector $x_{i,j}$.

%More generally let $\mathscr T$ be a tableau obtained from $\mathscr D$ in which the integers increase down the columns and along the rows from left to right.

 $\mathscr T$ need not be a standard tableau because column heights need not be increasing, in other words there may be gaps in the rows of $\mathscr T$. Yet $\mathscr T$ becomes a standard tableau when we shift boxes in a given row from right to left to close up gaps. Indeed the content of each row does not change and in the new tableau entries still increase down the columns. Then through the Robinson-Schensted correspondence and a result of Spaltenstein, $\mathscr T$ corresponds to an orbital variety for $\mathfrak {sl}(n)$ and every orbital variety is so obtained exactly once - see \cite [2.1]{JM}, for example.

One may recover $\mathscr D$ from $\mathscr T$ by forgetting the entries.

\subsection{}\label{2.4}

As in \cite [4.1.2]{FJ1} we say that two columns of height $s$ of $\mathscr T_\mathscr m$ are neighbouring if there are no columns of height $s$ strictly between them. To any pair of neighbouring columns $C_v,C_{v'}$, there is a Benlolo-Sanderson invariant $M(v,v')$ \cite [4.1.3]{FJ1} which is a polynomial in the co-ordinate functions $x_{i,j}$, where $\{i,j|i<j\}$ are the entries of the columns between $C_v,C_{v'}$. We will recall the description of this invariant in \ref {2.6}.

The zero set of $M(v,v')$ is a hypersurface orbital variety $\mathscr V(v,v')$ in $\mathfrak m_{\pi'}$. It is described in \cite [2.6]{FJ2} as a $B$ saturation set, specifically as $\overline{B.\mathfrak u_{v,v'}}$,
with $\mathfrak u_{v,v'}=\mathfrak n \cap w_{v.v'}(\mathfrak n)$ for a suitable choice of $w_{v,v'} \in W$.

\subsection{}\label{2.5}

In the above $w_{v,v'}$ is obtained by modifying \cite [2.4]{FJ2} the tableau $\mathscr T_\mathfrak m$. Here the details will be briefly recalled below. In this the columns outside $C_v,C_{v'}$, play no role.

Being neighbouring columns of height $s$, means that there are no columns of height $s$ strictly between $C_v,C_{v'}$. Let $m$ be the entry in the box $C_{v'} \cap R_s$.

Then as in \cite [2.6]{FJ2}, \textit{ignoring columns of height $<s$} shift $m$ into the rightmost column of height $\geq s$ and shift simultaneously the parts of the columns strictly below $R_s$, one position to the left.

This defines the modified tableau $\mathscr T_\mathfrak m(v,v')$. For example, if every column strictly between $C_v,C_{v'}$ has height strictly less than $s$, then $\mathscr T_\mathfrak m(v,v')$ is obtained by removing $m$ from $C_{v'}$ and placing it in the box $R_{s+1} \cap C_v$.

Now define $w_{v,v'}$ in word form by reading the entries of $\mathscr T_\mathfrak m(n)$ from bottom to top and then from left to right.

One always has $\mathfrak u_{v,v'} \subset \mathfrak m$ via \cite [Lemma 2.5]{FJ2}.

If there are no columns of height $>s$ between $C_v,C_{v'}$, then in the matrix presentation of $\mathfrak m$ we obtain $\mathfrak u_{v,v'}$ by removing the last column from $\mathfrak m$, when $C_{v'}$ is the last column of $\mathscr D$ and if not in a suitable truncation of $\mathfrak m$.

The general description of $\mathfrak u_{v,v'}$ is more complicated. It is determined through the root vectors of $\mathfrak n$ that do not appear in $\mathfrak u$ - called the excluded root vectors. They are given in \cite [2.7]{FJ2}.

\subsection{}\label{2.6}

Let $C_v, C_{v'}$ are neighbouring columns of height $s$. If they are not the first and last columns of $\mathscr D$, then we may truncate $\textbf{M}$ to make this so. (See also the reduction in \cite [3.2]{FJ1}.)

Let $M_s(v,v')$ be the $n-s \times n-s$ minor in the bottom left hand corner of the truncated \textbf{M}.

Let $P^-$ denote the opposed parabolic to $P$. Then $M_s(v,v')$ is a $P^-$ semi-invariant for co-adjoint action.

We consider $M_s(v,v')$ as a polynomial function on $\mathfrak m+\Id$, through the Killing form. In particular it is not homogeneous and it is \textit{not} necessarily a $P'$ invariant for adjoint action. Remarkably its leading term $\gr M_s(v,v')$ is a $P'$ invariant \cite {BS} for adjoint action, which we call the Benlolo-Sanderson invariant\footnote{A suggestion as to how this should generalize for all types is given in \cite [3.2, 3.5, 3.6]{FJ1}.}. As noted in \cite {BS}, its degree $d(\gr M_s(v,v'))$ is given by
$$d(\gr M_s(v,v'))= \sum_{i=v}^{v'-1} \min (c_i,s).$$

 The following is proved in \cite [2.6]{FJ2} by an argument similar to that of \cite [Prop. 3.8]{JM}.

\begin {prop} The closure of $B.(\mathfrak n \cap w_{v,v'}(\mathfrak n))$ is the zero variety of $\gr M_s(v,v')$ in $\mathfrak m$.
\end {prop}

\subsection{}\label{2.7}

Fix a parabolic $\mathfrak p$ and as in \ref {2.3}, let $\mathscr D$ denote the diagram it defines. Let $\mathscr P$ be the set of all pairs of neighbouring columns $C_v,C_{v'}$ of $\mathscr D$. For ease of noatation we write $C_v,C_{v'}$ as $v,v'$.

By Proposition \ref {2.6} and \cite [4.1.2]{FJ1}, the nilfibre $\mathscr N$ for the action of $\mathfrak p$ on its nilradical is given by
$$ \mathscr N = \cap_{{v,v'}\in \mathscr P} \overline {B. \mathfrak u_{v,v'}}.$$

Obviously this contains
$$\mathscr C:= \overline {B. \cap_{{v,v'}\in \mathscr P}\mathfrak u_{v,v'}}.$$

We showed in \cite [6.9.8]{FJ2} that rather remarkably, the latter is an irreducible component of $\mathscr N$ and is just the canonical component $\mathscr N^e$ referred to in \ref {1.4}.

\section{The Description of the Canonical Weierstrass Section }\label{3}

\subsection{}\label{3.1} Start from the tableau $\mathscr T$ defined by the nilradical $\mathfrak m$ of a standard parabolic as in \ref {2.3}.

 We may describe a (rather particular) linear subvariety $e+V$ of $\mathfrak m$ by a family of lines $\ell_{i,j}$ joining boxes in $\mathscr T$ with labels $i<j$.

 By \ref {2.3}, the line $\ell_{i,j}$ defines a co-ordinate vector $x_{i,j} \in \mathfrak m$.

 Then $e+V$ is defined by two families of lines. Those of the first family carry the label $1$ and specify the sum of root vectors giving $e$ and those of the second family are labelled by $\ast$ and specify the sum of root subspaces giving $V$.

 \

 \textbf{N.B.} We say interchangeably that a line $\ell$ is labelled by a $1$ (or a $\ast$) or carries a $1$ (or a $\ast$).

 \

 In this we call the support of $e$, denoted $\supp{e}$, (resp. support of $V$, denoted $\supp{V}$) those co-ordinate vectors (resp. subspaces) defined by the lines carrying a $1$ (resp. a $\ast$).

 A line carrying a $\ast$ may be ``gated'' (see \cite [4.1]{FJ2}.

 The actual choice of the labelled lines defining $e+V$ leading to it being a Weierstrass section is via a complicated combinatorial construction which is nevertheless essentially canonical. The details are given in \cite [Sect. 5]{FJ2}. They will not be repeated here, but in \ref {3.4} we describe the resulting lines by an algorithm which is surprisingly simple.

 In \ref {3.6} we show that these lines with their labels can be recovered very simply from the image of a map from compositions to the set of distinct non-negative integers together with the set of column heights. This ``composition'' map may well be of independent interest.

 Recall the notation of \ref {2.1}. We may view $e+V$ as a linear subvariety of $\mathfrak m$, by putting $1$ (resp. $\ast$) at the $(i,j)^{th}$ place of $\textbf{M}$, whenever in our previous convention the line $\ell_{i,j}$ carries a $1$ (resp. $\ast$).

 We call $e+V$ the canonical Weierstrass section, since it has a uniform construction for all parabolics.

\subsection{}\label{3.2}
\subsubsection{}\label{3.2.1}

%Recall the presentation of a parabolic $\mathfrak p$ given in \ref {2.3}

Following \ref {2.3}, let $\mathscr C=(C_1,C_2,\ldots,C_k)$ be the set of columns in $\mathscr D$ defining the Levi factor of a parabolic subalgebra $ \mathfrak p$ (in $\mathfrak {sl}(n)$), whose block sizes are the $\{c_i\}_{i=1}^k$.

%Then $\mathfrak p$ is defined where $c_i$ equals ht $C_i$.

Let $\mathscr T$ be the tableau obtained from $\mathscr D$ by the prescription given in \ref {2.3}.

 We denote the set of labelled lines in $\mathscr T$ between columns in $\mathscr C$ given through \cite [5.4]{FJ2} by $\ell(\mathscr C)$.

\subsubsection{}\label{3.2.2}

By \cite [5.4.9]{FJ2} $\ell(\mathscr C)$ behaves well under removing (or adjoining) columns on the right.

Indeed let $S$ denote the set of (ordered) columns obtained from $\mathscr C$ by deleting the last column $C_k$. Then

\begin {lemma} $\ell(S)$ is obtained from $\ell(\mathscr C)$ by deleting $C_k$ and all the lines from $S$ to $C_k$.
\end {lemma}

\subsubsection{}\label{3.2.3}

This fails for removal of columns from the left but we have the next best thing described below.

To make the comparison with \ref {3.2.2} less confusing we adjoin a column $C_0$ on the left of $\mathscr C$ and denote the resulting set of (ordered) columns by $T$.

\begin {lemma} Suppose $C_0$ has height $i$. Concerning only the lines lying entirely in $R^{i-1}$, one obtains $\ell(\mathscr C)$ from $\ell(T)$ by deleting $C_0$ and all the lines from $\mathscr C$ to $C_0$.

\end {lemma}

This is proved in a slightly stronger form in \cite [5.4.10]{FJ2}.

\subsection{}\label{3.3}

To complete our description of $e$ we need to describe how our construction adjoins lines in the last column to the earlier ones. This is achieved in the following subsections.

Let $b_{i,j}$ denote the box in $\mathscr D$ lying in $R_i\cap C_j$. We may also denote a box by its entry $l$, that is as $b(l)$.

\subsubsection {}\label {3.3.1}

Recall the notion of \ref {3.2}. In particular $C_k$ denotes the last column of $\mathscr D$. Recall the notation of \ref {2.1}, \ref {2.2}.

 Since we are considering the height $s$ of $C_k$ to be a variable (taking non-negative integer values), it is convenient to write $C_k$ (resp. $\textbf{C}_k$) more precisely as $C_k(s)$ (resp. $\textbf{C}_k(s)$) when this height is $s\in \mathbb N$. Similarly we write $\textbf{M}$ as $\textbf{M}(s)$, when its last column block is $\textbf{C}_k(s)$.

 We may note that $n$ is the largest entry of $C_k$ and occurs in its lowest row.

Recall that $S$ denotes the set of columns (resp. column blocks) excluding the last, that is $S=\{C_1,C_2,\ldots,C_{k-1}\}$. Let $\textbf{S}$ denote the corresponding set of column blocks, that is $\{\textbf{C}_1,\textbf{C}_2,\ldots,\textbf{C}_{k-1}\}$.
%By Lemma \ref {3.2.2} the entries of $\textbf{S}$ do not change as $s$ is increased.

Recall \ref {3.1} that $\textbf{M}$ has entries in $\{1,*\}$ which determine the Weierstrass section $e+V$. These entries all occur in the sub-matrix defining $\mathfrak m$.

%Here the $1$'s (resp. $0$'s) occur at the co-ordinates determining $e$ (resp. $V$).

Let $\textbf{M}^k$ be the $n-s \times n-s$ matrix formed by omitting the last $s$ rows and columns of $\textbf{M}$. The blocks $\{B_i\}_{i=1}^{k-1}$ (resp. the column blocks $\{\textbf{C}_i\}_{i=2}^{k-1}$) define the Levi factor $\mathfrak l^\kappa$ and nilradical $\mathfrak m^\kappa$ of a standard parabolic subalgebra $\mathfrak p^\kappa$ of $\mathfrak {sl(n-s)}$.

By Lemma \ref {3.2.2}, the entries in $\textbf{S}$ are unchanged if we omit the last column $C_k(s)$ or increase its height $s$. Thus if we define $e^\kappa$ (resp. $V^\kappa $) by dropping all terms which obtain from lines joining $C_k$ to $S$, then $e^\kappa+V^\kappa$ is a Weierstrass section for the action of the derived algebra of $\mathfrak p^\kappa$ on $\mathfrak m^\kappa$.

\

\textbf{N.B.}. The notation $C_k$ was already used in \cite [2.1]{FJ2} for the last column of $\mathscr D$. However in Sections \ref {4.4} and \ref {4.5}. it also convenient to use $k$ as a dummy index for the co-ordinates. This could cause some confusion (for example in Lemma \ref {4.4.3}). Thus we have used $\kappa$ instead of $k$ as the superscript above.

\subsubsection {}\label {3.3.2}

Let $\ell_{b,b'}$ be a line given by the construction of a Weierstrass section, that is by \cite [5.4]{FJ2}. It may carry a $1$ or a $\ast$. Suppose $b$ occurs in $R_i$ and labelled in $\mathscr D$ by $r$. Then a $1$ or a $\ast$ appears in the $r^{th}$ row of $\textbf{M}$. Yet $i,r$ are only vaguely related, so it is convenient to denote the row in which a $1$ or a $\ast$ appears in $\textbf{M}$ using boldface to avoid confusion, that is by $\textbf{r}$, which will always refer to a row of \textbf{M}.

\

Let $C,C'$ be a pair of neighbouring columns. Then $\ast$ always appears in the rightmost column of the column block $\textbf{C}'$, \cite [6.6.2]{FJ2}. The row of $\textbf{M}$ in which it appears is also specified in \cite [6.6.2]{FJ2}, however this is more complicated. It is through the canonical row associated to the pair which will now be denoted in boldface, that is by $\textbf{r}_{C,C'}$. It is the row of $\textbf{M}$ in which $\ast$ ought to appear and in fact does appear if the pair $C,C'$ does not ``surround'' any other pair of neighbouring columns in the sense of \cite [6.6.2]{FJ2}, as recalled below.

We say that a pair $C_1,C_1'$ of neighbouring columns surround the pair $C,C'$, if ht$C_1'= \text{ht} C +1$, and $C_1'$ lies strictly to the right of $C'$ with no columns of height $\geq s$ between. We do \textit{not} require $C_1$ to lie to the left of $C$ - see examples below. In this case and if there is a pair $C_1,C_1'$ which do surround the pair $C,C'$, then the $\ast$ corresponding to the latter pair again occurs in $\textbf{r}_{C,C'}$.

This process gives rise to a string of $\ast$'s in $\textbf{r}_{C,C'}$. This string may or may not end in a $1$ - see \cite [6.6.2, Examples] {FJ2}.

\

One may ask how $\textbf{r}_{C,C'}$ and $\textbf{r}_{C_1,C_1'}$ are related. The latter is strictly less than the former (that is to say the second $\ast$ is pushed downwards in \textbf{M}) if $C_1$ lies to the left of $C$, for example in the composition $(2,1,1,2)$. Yet the latter is strictly greater than the former (that is to say the second $\ast$ is pushed upwards in \textbf{M}) if $C_1$ lies to right of $C$, for example in the composition $(1,2,1,2)$.

 \subsection {Increasing the height $s$ of $C_k$}\label {3.4}

 Given a column $C$ of height $s$ and $t\leq s$, let $\textbf{C}^t$ denote the $t^{th}$ column of the corresponding column block $\textbf{C}$. Again the labelling of the columns in $\mathscr D$ and in $\textbf{M}$ is different but related.

 \subsubsection {Two Preliminaries}\label {3.4.1}

Recall the notation of \ref {3.3.1}. By \ref {3.3.2}, if $C_k(s)$ has a left neighbour in $S$, then $\textbf{C}_k(s)$ has just one $*$ entry and this occurs in $\textbf{C}_k(s)^s$, which is the rightmost column of $\textbf{C}_k(s)$, and has no $\ast$ entries otherwise. Again by \cite [Lemma 5.4.8(vi)] {FJ2} every column and every row of $\textbf{C}_k(s)$ has at most a single $1$ entry.

Recall \cite [5.4.1]{FJ2} that a right extremal box of $S$ is one with no adjacent right neighbour in $\mathscr C$. Such a box exists exactly when $S$ admits a rightmost column $C$ of height $>s$. In (b) below, we let $b$ denote the box $R_{s+1}\cap C$, which is clearly right extremal and denote its entry by $m$.

Suppose $C_k$ has height one. This means that $\textbf{C}_k$ is a single column. Let $t$ be the height of $C_{k-1}$. Then either

\

(a). If $S$ has no column of height one, then the only entry of $\textbf{C}_k$ is $1$ on its $(n-t)^{th}$ row.

\

(b). If $S$ has a column of height one, then $\textbf{C}_k$ has an $*$ on its $(n-t)^{th}$ row. Moreover with $b,m$ given as above, then $\textbf{C}_k$ has a $1$ on row \textbf{m} (in its unique column).

%In addition a $1$ may appear on another row following the rules given in (iv) below.

\

For (a) and the first part of (b), observe that step two of \cite [5.3]{FJ2} gives a horizontal line $\ell$ carrying a $1$ (resp $\ast$) in the first row $R_1$ joining the highest box $b'$ in $C_{k-1}$ to a unique box in $C_k$.
Moreover by \cite [5.4]{FJ2} in step $3$, the horizontal lines in $R_1$ and their labels are left unchanged.

The last part of (b) follows from \cite [5.4.7(i)]{FJ2}, which gives a line $\ell_{b,b_{s,k}} \in \ell(\mathscr D)$ carrying a $1$, if there is a right extremal box $b$ in $S\cap R_{s+1}$.

\subsubsection {}\label {3.4.2}

As we increase the height $s$ of $C_k$, then by Lemma \ref {3.2.2}, the only changes which occur are in the entries of $\textbf{C}_k$ (and not in the entries of $\textbf{S}$). Moreover these behave according to a surprisingly simple set of rules which we now describe.

The proofs will be given in \ref{3.5}.

\

 (i). Suppose $t \in [1,s-1]$. The entries of
 $\textbf{C}_k(s+1)^t$ are those of $\textbf{C}_k(s)^t$.

 \

 We now describe the entries of the last two columns of $\textbf{C}_k(s+1)$, that is when $t$ above is in $\{s,s+1\}$.

 \

 (ii). Suppose that $S$ has a column of height $s$. Then there is a $\ast$ in $\textbf{C}_k(s)^s$. It is replaced by a $1$ in $\textbf{C}_k(s+1)^s$, in the same row.

 \

 In other words the position of $\ast$ does not change but it becomes a $1$.

\

(iii). Suppose that $1$ appears in $\textbf{C}_k(s)^s$ in some row $\textbf{r}$ of $\textbf{M}$, then either (A) or (B) below holds.

\

(A). Suppose there is a $\ast$ in $\textbf{C}_k(s)^s$ and in row $\textbf{r}'$ (different of course to $\textbf{r}$).

(This occurs exactly when $S$ has a column of height $s$ and a column of height $>s$. In this the line $\ell$ carrying $\ast$ is gated and the line $\ell'$ carrying $1$ is constructed by \cite [5.4.7(i) or(ii)]{FJ2} at stage $s+1$ and both lines have $b_{s,k}$ as their right hand end point.)

Then a $\ast$ (resp. $1$) appears in $\textbf{C}_k(s+1)^{s+1}$ and in row $\textbf{r}$ if $S$ has (resp. does not have) a column of height $s+1$.

\

In other words $1$ is shifted to the right becoming a $\ast$ or a $1$.

\

(B). Suppose there is no $\ast$ in $\textbf{C}_k(s)$. This occurs exactly when $S$ has no column of height $s$.

Then a $1$ appears in $\textbf{C}_k(s+1)^s$ in row $\textbf{r}$.

\

In other words this $1$ entry stays fixed as $s$ increases by one.

\

If $S$ admits a column of height $s+1$, a $\ast$ appears $\textbf{C}_k(s+1)^{s+1}$, at the end of the lowest string (cf \ref {3.3.2}) of $\ast$'s (not ending in a $1$). It is replaced by a $1$ at that place if $S$ meets $R_{s+1}$, but does not admit a column of height $s+1$.

\

In other words a string of $\ast$'s on a given row of $\textbf{M}$ may be augmented in its last column by a $\ast$ or a $1$, as prescribed.

%This conclusion is similar to that of the last part of $A$.

\

\textbf{Examples}. Consider the composition $(3,2,1,1,2,2)$. Then the line $\ell:=\ell_{9,11}$ carries a $\ast$ and the line $\ell':=\ell_{6,11}$ (obtained by the joining of lose ends \cite [5.4.7(ii)]{FJ2}) carries a $1$. Thus $A$ holds.

Increasing the height of the last column by one, gives a line (namely $\ell_{6,12})$ carrying a $\ast$. Had the first column been of height four, this line would have carried a $1$.

By contrast for the composition $(1,2,1)$, the line $\ell:=\ell_{2,4}$ carries a $\ast$ and the line $\ell':=\ell_{3,4}$ (obtained by \cite [5.4.7(i)]{FJ2}) carries a $1$. Increasing the height of the last column by one, gives a line (namely $\ell_{3,5})$ carrying a $\ast$. Had the first column been of height three, this line would have carried a $1$.

\

 Consider the composition $(3,1,1,2)$. There is no $\ast$ in the last column; but there is a $1$ given by the line $\ell_{4,7}$. Thus $B$ holds.

 When the last column is increased in height by one, a $\ast$ appears in $\textbf{C}_k(3)^3$. It is replaced by a $1$ if the first column is increased in height by one.

\

 Suppose that $s$ has been increased by one and (i)-(iii) have be carried out.

 \

 \textbf{N.B.} Recall \cite [5.4.8(viii)]{FJ2} that a $1$ in a given row of \textbf{M} cannot be followed by a $\ast$ in a subsequent column.

 \

 Suppose $S$ does have a column of height $s+1$. Then (for $s\geq 1$) the rules (i)-(iii) may not give a $1$ in the last column of $\textbf{M}(s+1)$.

 In this case the construction of \cite [5.4]{FJ2} can add a $1$ to the last column of $\textbf{C}_k(s+1)$ and this follows the rule (iv) given below.

 %Then some $1$'s appear in the last column of $\textbf{C}_k(s+1)$ but these may not be all!

 Let $\mathscr R$ denote the set of rows of $\textbf{M}(s+1)$ either not having a $1$ except perhaps in the last column or not having a $\ast$ in its last column.

 Now recall that the elements of Weyl group of the Levi factor of $P$ interchanges the rows of $\textbf{M}(s+1)$ and so induce equivalence classes in $\mathscr R$. If $\mathscr R$ is not empty, let $\textbf{r}' $ be the unique lowest row in $\mathscr R$ (i.e. that with the largest subscript) and let $\textbf{r}$ be the unique highest row in the equivalence class of $\textbf{r}'$.

 \

 (iv). Suppose $\mathscr R$ is not empty. Then $1$ is placed in $\textbf{C}_k(s+1)^{s+1}$ in row $\textbf{r}$, \textit{or already occurs there} through $(i)-(iii)$.

 \

 Thus up to the action of the Levi, $1$ in the last column of $\text{M}(s+1)$ occurs in the row below the empty rows.

 \

 \textbf{Remark 1.} The complication of having to consider equivalence classes comes about because our construction though canonical is by induction on rows in which the highest free box is linked. It is perfectly possible that we would obtain a Weierstrass section $e+V$ with $e$ regular by putting $1$ in the intersection of the last column of $\textbf{C}_k(s+1)$ with the lowest row of $\mathscr R$, that is with $\textbf{r}'$.

 \

 \textbf{Remark 2.} In the subsequent application of (iv), the above \textbf {N.B.} is of prime importance.

\

\textbf{Remark 3.} Our presentation is perhaps a convoluted way of saying that if $1$ already appears in the last column by (i)-(iii) above, then it must also follow the rule, set out by (iv).

\

(v). If the last column of $\textbf{C}_k(s)$ is empty, so are the last two columns of $\textbf{C}_k(s+1)$.

\subsection {}\label {3.5}

In order to prove the assertions in \ref {3.4}, we must recall the construction in \cite [Sect. 5.4]{FJ2}. Unfortunately the reader may have to study loc. cit. first \textit{but we shall give explicit references to every required result}.

%For any column $C_j$ and any row $R_i$ let $b_{i,j}$ denote the box in $R_i\cap C_j$.

Recall the notation of \ref {3.3}.

Recall that $s$ is the height of $C_k(s)$.

%Recall the stages in \cite [5.4]{FJ2}.

On increasing $s$ by one, a left going horizontal line from $b_{s+1,k}$ may be introduced at step $2$.

Then we apply the $(s+1)^{th}$ stage of \cite [5.4]{FJ2}. Briefly this does just three things.

\

$(\textbf{X})$. (Some) horizontal lines in $R_{s+1}$, required by step $2$, are deleted \cite [5.4.6]{FJ2}.

\

$(\textbf{Y})$. (Some) lines in $R^{s}\setminus R^{s-1}$ carrying a $1$ given by applying \cite [5.4.7(ii)]{FJ2} at stage $s$, are deleted. (These are the lines given by ``joining loose ends''.) No lines in $R^{s}$ carrying a $\ast$ are deleted, but may be relabelled by a $1$.

\

$(\textbf{Z})$. New lines in $R^{s+1}\setminus R^s$ are introduced according to the rule in \cite [5.4.7]{FJ2}. They may carry either a $1$ or a $\ast$.

\

Recall that since only the last column $C_k$ is being changed, only lines meeting $C_k$ are changed (Lemma \ref {3.2.2}).

%Then a line carrying a $\ast$ (meeting $C_k(s)$) is relabelled with a $1$ in $C_k(s+1)$.

\

\textbf{Example.} Let $\mathscr D$ be given by the composition $(2,1,1,1)$, so then $k=4,s=1$.

The only lines joining $C_4$ to $S$ are $\ell_{4,5}$ which carries a $\ast$ and $\ell_{3,5}$ which carries a $1$ (the latter joins loose ends).

When $s$ is increased by one, a horizontal line $\ell_{2,6}$ is introduced at step $2$.

At stage $2$, $\ell_{3,5}$ is deleted and the $\ast$ on $\ell_{4,5}$ is replaced by a $1$, illustrating $(\textbf{Y})$ and $\ell_{2,6}$ is deleted, illustrating $(\textbf{X})$. The latter is replaced by $\ell_{3,6}$, which carries a $\ast$, illustrating $(\textbf{Z})$. (Strictly speaking $\ell_{2,4}$ also replaces $\ell_{2,6}$ but is already present when $s=2$ as predicted by Lemma \ref {3.2.2}.)

\

It remains to consider the assertions of \ref {3.4.2}.

\

Consider $(i)$.

By \cite [5.4.8 (iii),(ix)]{FJ2} a right going line can only go down by one row and then it must carry a $1$ and $b'$ must also be the right end point of a line carrying a $\ast$. Then by \cite [6.6.2, second paragraph]{FJ2} $b'$ must be the box $C_k(s)\cap R_s$. Thus

\

$(*)$. Every line with right end point in $R^{s-1}\cap C_k(s)$ lies entirely in $R^{s-1}$.

\

 On the other hand when the height of $C_k$ is increased from $s$ to $s+1$ and stage $s+1$ is implemented, the lines lying entirely in $R^{s-1}$ are left unchanged by $(\textbf{X})-(\textbf{Z})$. Then by $(*)$, the lines joined to a $b_{t,k}:t<s$ are unchanged (and so are their labels).

This gives the assertion of (i).

\

Consider (ii).

By hypothesis $S$ has a column of height $s$.

Then the left going line $\ell_{b,b_{s,k}}$ from $b_{s,k}$ to that box $b\in S$ determined in stage $s$ by \cite [5.4.7(i)]{FJ2}, carries a $\ast$ by \cite [5.4.8]{FJ2}. Moreover it lies entirely in $R^s$, so by $(\textbf{Y})$, it is not deleted but now carries a $1$ because $C_k(s+1)$ has height $s+1$. This establishes (ii).

\

Consider (iii)A.

The hypothesis means that there is a line $\ell$ from $b\in S$ to $b_{s,k} \in C_k(s)$ carrying a $\ast$ and a second line $\ell'$ from $b' \in S$ to $b_{s,k} \in C_k(s)$ carrying a $1$. This means that $\ell$ must be gated and the construction in \cite [5.4.7(i)or(ii)]{FJ2} gives a line $\ell'$ carrying a $1$ whose left end point lies in $S$.

 As $s$ is increased by one, $\ell$ remains with the same end points but now carries a $1$ (as noted in (ii)) so can no longer be gated. This forces $\ell'$ to be replaced by a right going line $\ell''$ to $b_{s+1,k}$, which of course is in the same column of $\mathscr D$ as $b_{s,k}$. Then by the minimal distance criterion of \cite [5.4.7]{FJ2}, the line $\ell''$ must have the same left-hand point $b'$ as \textit{had} $\ell'$. It carries a $\ast$ (resp. $1$) if $S$ has (resp. does not have) a column of height $s+1$. The new entry to which $\ell''$ gives rise, lies in the same row of $\textbf{M}$, yet now in the last column of $C_k(s+1)$, because its right hand end point is $b_{s+1,k}$.

 This establishes the last part of $(iii)A$.

% to say it has been shifted by one column to the right from the previous place given by the erstwhile $\ell'$.

\

\textbf{Examples.} In the case of the parabolic defined by the composition $(1,2,2,1)$, there is a gated line $\ell=\ell_{4,6}$. Then $\ell'=\ell_{5,6}$ being given by (i) of \cite [5.4.7]{FJ2}. When the last column is increased by one in height, $\ell$ carries a $1$ and $\ell'$ is replaced by a line $\ell''=\ell_{5,7}$ and in this case carries a $\ast$.

On the other hand for the composition $(2,3,1,3,2,1,2)$ the gated line $\ell$ is that joining $(10,14)$. There is a further gated line joining $(8,11)$. Then $\ell'$ is given by \cite [5.4.7(ii)]{FJ2} and joins $(8,14)$. When the last column is increased in height by one, $\ell$ carries a $1$ and $\ell'$ is replaced by a line $\ell''$ with the same beginning point as $\ell'$ but ending in $b_{s+1,k}$, that is joining $(8,15)$ and in this case carries a $\ast$.

\

Consider (iii)B.

Under the hypothesis there is no line going from $S$ to $b_{s,k}$ carrying a $\ast$. Consequently the line $\ell$ from $b\in S$ to $b_{s,k} \in C_k(s)$ carrying a $1$ is obtained at stage $s$ and moreover by using (i) and not by (ii) of \cite [5.4.7]{FJ2}. It lies entirely in $R^s$, though it need not be horizontal (viz. $\ell_{4,7}$ of the example below). Thus it remains at stage $s+1$ by $(\textbf{Y})$.

This proves the first part of (iii)B.

\

\textbf{Example.} Consider the line joining $(4,7)$ in the composition $(3,1,1,2)$ obtained by applying \cite [5.4.7(i)]{FJ2} at stage $2$. It is carries a $1$ since there is no column in $S$ of height $2$. When $s$ is increased by $1$, it remains.

\

Consider the second part of (iii)(B).

Under the hypothesis, at step $2$ there is a left going horizontal line $\ell$ from $b_{s+1,k}$ if $S$ meets $R_{s+1}$ and no such line otherwise. In the former case it carries a $\ast$ (resp. a $1$) if $S$ admits (resp. does not admit) a column of height $s+1$. In step $3$, the right end point of $\ell$ is unchanged by the way boxes are rejoined \cite [5.4.7(i)]{FJ2} and the labelling is left unchanged \cite [5.4.8]{FJ2}. Thus a $\ast$ (resp. $1$) appears in $\textbf{C}_k(s+1)^{s+1}$, as asserted.

\

\textbf{Example.} Consider again the composition $(3,1,1,2)$.
%Then $\ell$ is line joining $(4,7)$ and carries a $1$. It occurs at the end of a row of $\ast$'s (actually just one at $x_{4,5}$) in $R_4$.
When $s$ is increased by $1$, the new line $\ell'$ joins $(3,8)$ and carries a $\ast$. Had the first column had height $>3$, then $\ell'$ would carry a $1$. Either appear in $R_3$ at the end of the lowest string of $\ast$'s not ending in $1$ (which is actually an empty string in this case).

\

Consider (iv). (This is the most difficult part.)

%For the convenience of the reader we partly recall some statements already made.

%Suppose after (i)-(iii) have been carried out, there is no $1$ in the last column of $C_k(s+1)$ and $\mathscr R$ is not empty. Then

There can be at most a single $1$ in $\textbf{C}_k(s+1)^{s+1}$. It corresponds to a line $\ell_{b,b_k(s+1)}$ in $\ell(S\sqcup C_k(s+1))$, carrying a $1$. Let $r$ be the entry of $b$. Then $1$ appears in row $\textbf{r}$ of \textbf{M}. Set $\mathscr R^-=\mathscr R\setminus \textbf{r}$.

Condition (iv) holds if every element of $\mathscr R^-\cap \textbf{C}_k(s+1)$ lies above $\textbf{r}$, or below but reached by the action of the Levi factor. This holds for example if $\mathscr R^-$ is empty.

Let $C_1,C_2,\ldots, C_{m-1}$ be the columns of height $s+1$ in $S$ and set $C_m:=C_k(s+1)$.

We now establish (iv) under

\

$(*)$. Suppose $S$ has no column of height $>s+1$.

\

\textbf{Claim}. Under $(*)$, $\mathscr R^-$ is empty.

\

Adjoin a column $C_0$ to the extreme left of $\mathscr D$ of height $s+1$. Then through the construction of \cite [5.4.7(i)]{FJ2} and the labelling in \cite [5.4.8]{FJ2}, there exists for all $i \in [1,m]$ a box $b'_i \in C_i\cap R_{s+1}$, a box $b_i \in R^{s+1}$, which by adjacency \cite [5.4.3]{FJ2}, must lie to the right of $C_{i-1}$ and a line $\ell_{b_i,b_i'}$ carrying a $\ast$.

On the other hand by \cite [5.4.8(v)]{FJ2} every box in $R^{s+1}$, with the exception (of course) of those in $C_m$ and the possible exception of the $b_i:i \in [1,m]$ has a right going line carrying a $1$.

When $C_0$ is suppressed, then $b_1\in C_0$ is also suppressed. If $b_1\notin C_0$, one has $\ell_{b_1,b_1'}\in \ell(\mathscr D)$ carrying a $1$, by the joining of loose ends \cite [5.4.7(ii)]{FJ2}, and likewise line $\ell_{b_i,b_{i+1}'}\in \ell(\mathscr D):i\in [2,m-1]$ carries a $1$ with the exception of $b_m:m\geq 2$, which carries a $\ast$ appearing in the last column of $\textbf{M}(s+1)$. Up to this exception every box $R^{s+1}$ hence in $\mathscr D$, not of course in $C_m$, has a right going line carrying a $1$. This proves the claim.

\textbf{Remark.} $\mathscr R$ itself is empty if $m = 1$. For the composition $(1,2,2)$ it is empty, for the composition $(1,1,2,2)$ it is not empty.

We now establish (iv) under

\

$(**)$. There is a column of height $>s+1$ in $S$.

\

Choose a column $C$ of height $>s+1$ in $S$ nearest to the right. Retain our previous notation for the columns of height $s+1$.

 Choose $t \in [1,m]$ minimal such that $C_t$ lies to the right of $C$. Set $b'=C\cap R_{s+2}$ which is an extremal box. By our chosen labelling $b_t'$ lies strictly to the right of $C$. If $t \geq 2$, then $\ell_{b_t,b_t'}$ is gated and since \cite [5.4.7(i)]{FJ2} is carried out before \cite [5.4.7(ii)]{FJ2}, there is a line $\ell_{b',b_t'}$ carrying a $1$ and this precludes the line $\ell_{b_{t-1},b_t'}$.

\

$(a)$. Suppose $t \leq m-2$.

\

 Then by the joining of loose ends \cite [5.4.7(ii)]{FJ2} there is a line $\ell_{b_{m-1},b_m'}$ carrying a $1$. Let $r$ be the entry of $b_{m-1}$. This puts a $1$ on the last column of $\textbf{M}(s+1)$ in row $\textbf{r}$.

 Yet any box with entry $u>r$ lies between the columns $C_{m-2},C_m$ and so (as in $(*)$) has a right going line carrying a $1$, with the exception of $\ell_{b_{m-1},b'_m}$ which carries a $\ast$ appearing in $\textbf{C}_k(s+1)^{s+1}$.

 This establishes (iv) if $(a)$ holds.

 \

 $(b)$. Suppose $t=m-1$.

 \

 The gated line $\ell_{b_{m-1},b'_{m-1}}$ cannot cross $C$ by adjacency \cite [5.4.8(vii)]{FJ2}. Thus $b_{m-1}$ lies in a column to the right of $C$.

 Now by the joining of loose ends \cite [5.4.7(ii)]{FJ2} there is a line $\ell_{b_{m-1},b_m'}$ carrying a $1$. If $b_{m-1}$ lies in a column strictly to the right of $C$, then (iv) holds exactly as in $(a)$.

 Suppose $b_{m-1} \in C$. Then there may be empty rows meeting the last column of $\textbf{M}(s+1)$ below \textbf{r} indexed by the entries $r'>r$ of the boxes in $C$ strictly below $b'$. However these rows are equivalent to \textbf{r} via the action of the Levi factor.

 This establishes (iv) if $(b)$ holds.

 \

 $(c)$. Suppose $t=m:m>1$.

 \

 In this case since $\ell_{b_{m},b'_m}$ is gated, so there is a line $\ell_{b',b'_m}$ carrying a $1$ and we conclude as in $(b)$.

 This establishes (iv) if $(c)$ holds.

 \

 $(d)$. Suppose $t=m=1$.

 \

 Adjoin a column $C_0$ of height $s+1$ to the left of $\mathscr D$. Then by the joining of loose ends \cite [5.4.7(ii)]{FJ2}, there is a line $\ell_{b_1,b_1'}$ which carries a $\ast$ and therefore cannot cross $C$ by adjacency \cite [5.4.8(vii)]{FJ2}. Then by \cite [5.4.10]{FJ2} it remains with the same end points when $C_0$ is removed but $\ast$ is replaced by a $1$. Then we may conclude as in $(b)$.

 This completes the proof of (iv).

 \

 The following exemplifies $(*)$.

 \

\textbf{Example 1.} Take the column heights in $S\sqcup C_k(s)$ to be given by the composition $(1,2,2,2,1,2)$, with $s=1$.

In the example, $i=6,j=5$ and \cite [5.4.7(ii)]{FJ2} gives the line $\ell_{5,10}$ carrying a $1$. This line would not have been given by $(1)-(3)$ of \ref {3.4}.

\

The following exemplifies $(**)$ in case $(a)$.

\textbf{Example 2.} Consider the composition $(2,3,2,2,2)$, with $s=1$. Then $1$ occurs in row \textbf{8} and a $\ast$ in row $\textbf{9}$ both in the last column of $\textbf{M}(s+1)$. Through the joining of loose ends $1$ occurs in row \textbf{7}. This would not have been given by $(1)-(3)$ of \ref {3.4}.

\

The following exemplifies $(**)$ in case $(b)$, when $b_{m-1}$ lies in a column strictly to the right of $C$.

\textbf{Example 3.} Consider the composition $(3,1,1,2)$, with $s=1$. The line with right end point $b_{2,4}$ obtained from step $2$ is just $\ell_{2,7}$. Yet since $\ell_{4,5}$ is gated, this line is modified in step $3$ to $\ell_{4,7}$, via \cite [5.4.7(i)]{FJ2}. The resulting $1$ in \textbf{M} lies in row $\textbf{4}$, so below the empty row $\textbf{3}$.

 \

 The following exemplifies $(**)$ in case $(c)$.

Consider the composition $(4,2,1,3)$. The line with right end point $b_{3,4}$ obtained from step $2$ is just $\ell_{3,10}$. It is not modified in step $3$. The resulting $1$ in \textbf{M} lies in row $\textbf{3}$ so above the empty row $\textbf{4}$. Yet these rows are interchanged by that part of the Levi factor coming from the first column.

\

Consider (v). Let $|S|$ denote the greatest height of a column in $S$.

Observe that (v) implies that $b_{s,k}$ is a left extremal box and so $s>|S|$.

 Again (v) also means that there is no gated line $\ell_{b,b_{s-1,k}}$ in $R^{s-1}$ in stage $s$. Otherwise there would be a line $\ell_{b,b_{s,k}}$ carrying a $1$ by the very first part of \cite [5.4.7]{FJ2}. Moreover if $s>|S|$ then by \cite [5.4.8]{FJ2} no lines in $R^s$ carrying a $\ast$ are adjoined in stage $s+1$.

 By the first paragraph $s+1 >|S|$ and by the second paragraph there are
 no gated lines in $R^s$ at stage $s+1$. Thus the conclusion of (v) results by a second use of the very first part of \cite [5.4.7]{FJ2}.

\subsection {Stabilisation}\label {3.6}

%Let $|S|$ denote the largest height of a column of $S$.

\subsubsection {A Column Existence Lemma} \label {3.6.1}

The following is proved as part of ``Claim'' in the proof of \cite [Cor. 6.5.4]{FJ2}. For the convenience of the reader we shall repeat the argument.

\begin {lemma} Let $C_1<C_2$ be columns of $\mathscr D$ and a choice of boxes $b_i \in C_i:i=1,2$. Let $R_t$ be the row in which $b_2$ occurs. Assume that $\ell_{b_1,b_2}\in \ell(\mathscr D)$. Then there exists a column $C': C_1\leq C' <C_2$ of height at least t -1.
\end {lemma}

\begin {proof} This is immediate unless $b_1 \in R^{t-2}$.

If not, $\ell:=\ell_{b_1,b_2}$ must be left and downward going by at least two rows. This can only happen if there is a line $\ell'$ carrying a $\ast$, ungated at the $(t-1)^{th}$ stage \cite [5.4.3.5.4.4]{FJ2}, with the same left hand end-point $b_1$ as $\ell$, so allowing $\ell$ to be drawn \cite [5.4.7]{FJ2}. By up-going linkage \cite [5.4.3]{FJ2} the right hand end point $b'$ of $\ell'$ lies in $R_{t-1}$ and since $\ell'$ carries a $\ast$, the box $b'$ lies at the bottom of the column $C'$ in which it is contained, so $C'$ has height $t-1$. Yet $C_2$ has height $\geq t$ so by adjacency \cite [5.4.3]{FJ2}, $C'$ lies strictly to the left of $C_2$ and of course to the right of $C_1$.

\end {proof}

\subsubsection {}\label {3.6.2}

 Recall that a $\ast$ can only appear at most once in $\textbf{C}_k(s)$ and then necessarily in $\textbf{C}_k(s)^s$. Recall also that, for $t\leq s$, at most one $1$ can appear in $\textbf{C}_k(s)^t$, by \ref {3.3.2}. We shall use $\textbf{C}_k(s)^t=\phi$ to mean that neither a $\ast$ nor a $1$ appear in $\textbf{C}_k(s)^t$ and we say that this column is empty.

%Through \ref {3.4.2} (i)-(iii) one easily checks that

\begin {lemma} For all $t \leq s$ one has $\textbf{C}_k(s)^t=\phi$ if and only if $\textbf{C}_k(s+1)^t=\phi$.

\begin {proof} For $t<s$, this is immediate from \ref {3.4.2} (i). For $t=s$, ``only if'' follows from \ref {3.4.2} (v). For ``if'' suppose $\textbf{C}_k(s+1)^s=\phi$. Then by \ref {3.4.2}(ii), there is no $\ast$ in $\textbf{C}_k(s+1)^s=\phi$. Then by \ref {3.4.2}(iii), there is no $1$ in $\textbf{C}_k(s+1)^s=\phi$.
\end {proof}

\end {lemma}

\subsubsection {}\label {3.6.3}

\begin {lemma}

\

(i). Take $t>|S|$. Then $\textbf{C}_k(s)^{t+1}=\phi$, for all $s \in \mathbb N^+$.

\

(ii). Take $t\leq |S|$. Then $\textbf{C}_k(s)^t \neq \phi$, for all $s \in \mathbb N^+:s \geq t$.

\end {lemma}

\begin {proof}

%By \ref {5.4.2} (ii),(iii) one easily checks that $\textbf{C}_k(s)^s=\phi$ if and only if $\textbf{C}_k(s+1)^s=\phi$.
Consider (i).

Take $m$ to be an integer $\geq 0$. Recall (\ref {3.3}) that $b_{(|S|+m+2),k}$ is the lowest box in the last column $C_k(|S|+m+2)$ deemed to have height $|S|+m+2$.

If the last column of $\textbf{C}_k(|S|+m+2):m\geq 0$ is not empty, then there is a line $\ell_{b,b_{(|S|+m+2),k}}$ for some $b \in S$. Then by Lemma \ref {3.6.1}, there is a column of height $|S|+m+1$ in $S$. This contradiction proves that $\textbf{C}_k(|S|+m+2)^{|S|+m+2}=\phi$ for all $m \in \mathbb N$.

 Applying Lemma \ref {3.6.2}, we conclude that $\textbf{C}_k(s)^{|S|+m+2}=\phi$, for all $s\geq |S|+m+2$, whilst for $s<|S|+m+2$, the assertion is trivial. Finally the condition $t>|S|$, means that we can write $t=|S|+m+1$, for some $m \in \mathbb N$, whilst we have just shown that $\textbf{C}_k(s)^{|S|+m+2}=\phi$, for all $s \in \mathbb N$. Hence (i).

Consider (ii).

By definition $S$ has a column $C$ of height $|S|$. Choose that nearest to the right. If $s = |S|$, set $C'=C_k(s)$. Otherwise adjoin a column $C'$ of height $|S|$ to $\mathscr D$, on its right. By Lemma \ref {3.2.2} this does not change the left going lines from $C_k(s)$, nor their labels.

 Take $t\leq |S|$. Since $C_k(s)$ lies between $C,C'$ every box in $R^{|S|} \cap C_k(s)$ has a left going line by \cite [5.4.8(v)]{FJ2}. One may remark that they all carry a $1$ except if this line is left going from $C_k(s)^t: s=t=|S|$.

 Consequently $\textbf{C}_k(s)^t \neq \phi$ if $s\geq t$ and $t\leq |S|$, as required.

 %Then $C_k(t)$ has a neighbouring column in $S$ and so a $\ast$ appears in $\textbf{C}_k(t)^t$, which is therefore not empty. By Lemma \ref {3.6.2}, we obtain $\textbf{C}_k(s)^t$, for all $s\geq t$.

\end {proof}

\textbf{Remark.} Take $s\in \mathbb N^+$. Then $\textbf{C}_k(s)^{|S|+1}$ may or may not be empty.

\subsubsection {}\label {3.6.4}

\begin {lemma} For all $s>|S|$, the column block $\textbf{C}_k(s+1)$ obtains from $\textbf{C}_k(s)$ by adjoining an empty column on the right.
\end {lemma}

\begin {proof} By \ref {3.4.2}(i) one has $\textbf{C}_k(s+1)^u=\textbf{C}_k(s)^u$ for $u<s$. For $u=s$, this equality again holds by \ref {3.4.2}(v) if the right hand side is empty. Otherwise for $s>|S|$ it holds by \ref {3.4.2}(iii)B.

Finally by Lemma \ref {3.6.3}(i), one has $\textbf{C}_k(s+1)^{s+1}= \phi$, for all $s >|S|$.

% , it remains to show that $\textbf{C}_k(s+1)^s=\textbf{C}_k(s)^s$, for all $s >|S|$. This holds by \ref {3.4.2}(v), if the right hand side is empty. Otherwise it can only admit a $1$ in some row and then so does the right hand side by \ref {3.4.2}(iii)B.
\end {proof}

\subsubsection {}\label {3.6.5}

\begin {cor}

\

%$(i)$. For all $l\geq |S|$, the last column block $\textbf{C}_k(l+1)$ obtains from $\textbf{C}_k(|S|+1)$ by adjoining $l-|S|$ empty columns on the right, so defining an ``asymptotic'' last column block $\textbf{C}_k(\infty)$ which obtains from $\textbf{C}_k(|S|+1)$ by adjoining an infinite block of empty columns on the right.
%
%\

$(i)$. There exists an ``asymptotic'' last column block $\textbf{C}_k(\infty)$ which obtains from $\textbf{C}_k(|S|+1)$ by adjoining an infinite block of empty columns on the right.

\

(ii) $\textbf{C}_k(\infty)^t\neq \phi$, for all $t\leq |S|$ and has a $1$ in some unique row, $\textbf{r}_t$, which are moreover pairwise distinct.

\
\end {cor}

\begin {proof}

\end{proof}

(i). For all $l\geq |S|$, the last column block $\textbf{C}_k(l+1)$ obtains from $\textbf{C}_k(|S|+1)$ by adjoining $l-|S|$ empty columns on the right. Hence the assertion.

(ii). The first part follows from Lemma \ref {3.6.3}(ii). Since $S$ has no column of height $>|S|$, this non-trivial entry can only be a $1$. Yet by \cite [5.4.8(vi)]{FJ2} a $1$ can occur at most once on any given column (or row|) of $\textbf{C}$, for any column $C$ of $\mathscr D$. The first part gives the uniqueness of $\textbf{r}_t$, the second part that they are pairwise distinct.

\textbf{Remark.} Following the remark of \ref {3.6.3}, we note that
$\textbf{C}_k(\infty)^{|S|+1}$ need not be empty and then has a $1$ in some row $\textbf{r}_{|S|+1}$. Otherwise we set $\textbf{r}_{|S|+1}=0$.

 % The remaining non-trivial entries are all $1$'s occurring on distinct rows and columns.

 \subsubsection {}\label {3.6.6}

 The above result defines a map $S \mapsto \{\textbf{r}_t\}_{t=1}^{|S|+1}$ from compositions to the set of all pairwise distinct $|S|+1$ non-negative integers. We call it the composition map.

 %One would like to compute the composition map and to see if it arises elsewhere in combinatorics.

 %%Let $\mathscr S$ denote the set of heights of $S$ counted \textit{without} multiplicities.
%
% \begin {prop} The asymptotic last column block $\textbf{C}_k(\infty)$ determines $\textbf{C}_k(s)$, for all $s \in \mathbb N$.
% \end {prop}
%
% The proof will be given in the subsections below.

 \subsubsection {}\label {3.6.7}

 One might hope that the image of composition map determines the last column block $\textbf{C}_k(s)$ for all $s \in \mathbb N$. By definition it determines $\textbf{C}_k(s)$ for all $s|s >|S|$.

 To do this one should show that $\textbf{C}_k(s+1)$ determines $\textbf{C}_k(s)$. Yet this fails as the following example shows.

\

 \textbf{Example.}

 Consider the compositions $(2,1,1,3); (1,1,1,1,3)$ labelling the last column by $C_k(s+1):s=2$.

 In both cases the last column block $\textbf{C}_k(s+1)$ has a $1$ at co-ordinates $x_{4,5},x_{3,6}$ and its last column is empty.

 Now replace $s+1$ by $s$. In the first case the $x_{3,6}$ acquires a $\ast$, whilst in the second case there is no change.

 %In this the line $\ell_{3,6}$ carries a $1$ which is replaced by a $\ast$ in $C_k(s)$.

 %Yet for the composition $(1,1,1,3)$, with $C_k(s+1):s=2$, the line $\ell_{2,5}$ carries a $1$ which remains in $C_k(s)$. In both cases the last column of $\textbf{C}_k(s+1)$ is empty. Thus we need a knowledge $\mathscr S$ to separate these two cases.

% Let $\mathscr S$ denote the set of heights of $S$ counted \textit{without} multiplicities.
\subsubsection {}\label {3.6.8}

Yet we can realise the goal of \ref {3.6.7} by including information on column heights in $S$. This is described in the lemma below using the following notation.

 Let $\mathscr S$ denote the set of heights of $S$ counted \textit{without} multiplicities.

 Let $\textbf{r}_s$ be defined through the composition map (\ref {3.6.6}).

 Recall \ref {3.6.5}. In particular $\textbf{r}_{|S|+1}$ may or may not be empty.

 % Label the rows of the last column block $\textbf{C}_k(s)$ as in Corollary \ref {3.6.7}(ii).
%
% Recall \ref {3.4.2}(i) means that we only have to show that $\textbf{C}_k(s+1)$ determines the last column of $\textbf{C}_k(s)$.

 \begin {lemma}

 \
Take $s \leq |S|+1$.

 \

 (i). Suppose $s \in \mathscr S$. Then $\ast$ occurs in $\textbf{C}_k(s)^s$ in row $\textbf{r}_{s}$. Moreover $1$ occurs in $\textbf{C}_k(s)^s$ in row $\textbf{r}_{s+1}$, if and only if the latter is non-empty.

 % and if $\mathscr S$ has a column of height $> s$, then a $1$ occurs in $\textbf{C}_k(s)^s$ in row $\textbf{r}_{s+1}$,.

 \

 (ii). Suppose $s \notin \mathscr S$. Then $1$ occurs in $\textbf{C}_k(s)^s$ in row $\textbf{r}_{s}$.

 \

 In particular the image of the composition map and $\mathscr S$ determine the last column block $\textbf{C}_k(s)$, for all $s \in \mathbb N^+$.

 \end {lemma}

 \begin {proof} By \ref {3.4.2}(i), it is enough to show that the conclusions give the correct entries for $\textbf{C}_k(s+1)^t:t=s,s+1$ as described in \ref {3.4.2}.

 Consider (i).

 By {3.4.2}(ii) a $\ast$ in $\textbf{C}_k(s)^s$ gives a $1$ in $\textbf{C}_k(s+1)^s$ in the same row. Then by {3.4.2}(i), this entry remains in $\textbf{C}_k(\infty)^s$ again in the same row but becoming or staying a $1$. It therefore must belong to $\textbf{r}_s$, by definition of the latter. This gives the first part of (i).

 If a $1$ also occurs in $\textbf{C}_k(s)^s$, then \ref {3.4.2}(iii)A gives a $\ast$ (resp. a $1$) in $\textbf{C}_k(s+1)^{s+1}$ if $s+1 \in \mathscr S$ (resp. $s+1 \notin \mathscr S$), in the same row. In the first case by \ref {3.4.2}(i,ii), this entry remains in $\textbf{C}_k(\infty)^{s+1}$ again in the same row but becoming a $1$. It therefore must belong to $\textbf{r}_{s+1}$, by definition of the latter. Notice by \ref {3.6.3}(ii), $s+1 \in \mathscr S$ implies
 $\textbf{r}_{s+1}\neq \phi$.

 In the second case a $1$ appears in $\textbf{C}_k(s+1)^{s+1}$. Then by \ref {3.4.2}(i) and \ref {3.4.2}(iii)B, this $1$ remains in $\textbf{C}_k(\infty)^{s+1}$ again in the same row, thus in $\textbf{r}_{s+1}$ by definition of the latter. However in this second case it can happen that
 $\textbf{r}_{s+1} =\phi$. This gives a contradiction which implies that $1$ does not occur in $\textbf{C}_k(s)^s$.

 Similarly (ii) follows from {3.4.2}(iii)B.

 \end {proof}

 \textbf{Remarks.} Noting that in the example of \ref {3.6.7} that $S$ has no column of height $>s=2$, one may check that the conclusion of the lemma is in accordance with this example. Again consider the composition $(2,1,1,s)$. In this case $\textbf{r}_1=\textbf{4}, \textbf{r}_2=\textbf{3}, \textbf{r}_1=\phi$. Correspondingly $1,\ast$ occur in $\textbf{C}_4(1)^1$, whilst only $\ast$ occurs in $\textbf{C}_4(2)^2$.

 Considering the complexity of the rules in \ref {3.4.2}, the assertion of Lemma \ref {3.6.8} is extraordinarily simple.

 Clearly the composition map is of independent interest.

\section {Regularity} \label {4}

We continue to take $e+V$ to be the canonical Weierstrass section constructed in \cite [5.4]{FJ2} and described rather precisely in \ref {3.4} above.

We would have liked a solution which starting from the ordered set of heights of the columns directly gives the lines determining $e,V$.

This solution described in \ref {3.5} is perhaps not ideal, but it does behave well under adjoining of columns on the right Lemma \ref {3.2.2} and increasing the height of the last column. Then by Lemma \ref {3.6.8}, the solution is completely determined by the images of the composition maps and the set of column heights.

By \cite [Lemma 6.1]{FJ2}, $e$ belongs to the nilfibre $\mathscr N$.

We now would like to show that $e\in \mathscr N_{reg}$. Unfortunately this can fail because $e$ may give rise in VS pairs. This is explained in \ref {4.3} and \ref {4.4}.

\subsection {Pseudo-regularity} \label {4.1}

Recall that e is a sum of co-ordinate vectors $\{e_i\}_{1=1}^r$, none of which lie on the same row or same column of \textbf{M}. In particular the roots of the $\{e_i\}_{1=1}^r$ are linearly independent. Let $E$ denote the corresponding sum of root subspaces. As an immediate consequence of linear independence
$$\mathfrak h.e=E.\eqno{(1)}$$
Recall that $e\in \mathscr N$. Then via \cite [Lemma 3.1(ii)]{FJ2} the property that $e$ is regular, is equivalent to
$$\mathfrak p.e \oplus V= \mathfrak m.\eqno{(2)}$$
Consider $\sum_{i=1}^r \mathfrak p.e_i$. Obviously
$$\mathfrak p.E:= \sum_{i=1}^r \mathfrak p.e_i \supset \mathfrak p.e \eqno {(4)}$$
As we shall see the inclusion may be strict.

\

\textbf{Definition} $e$ is said to pseudo-regular if
$$\mathfrak p.E+V =\mathfrak m.\eqno{(5)}$$

\

It should be carefully noted that this is different to $e$ being quasi-regular. Indeed from the definition of pseudo-regularity it follows that $\dim \mathfrak p.E \geq \dim \mathfrak m -\dim V$ and this inequality can be strict. On the other hand quasi-regularity means that $\overline {P.E}=\mathscr N^c$ and as we shall see (Lemma \ref {4.5.3}(ii)) it implies that $\dim P.E =\dim \mathfrak m -\dim V$. Quasi-regularity generally fails and to recover it we must augment $E$ to a larger subspace $E_{VS}$ to be constructed in \ref {4.4.4} below.

\

Pseudo-regularity may be visualized as follows.

Consider a chess board with squares being the co-ordinate places of $\mathfrak m$ with those of $V$ removed. Partition this board by vertical (resp. horizontal) lines extending the sides of the blocks.

Consider $e$ as a set of rooks or castles\footnote{Apparently this strange dual terminology came about from a confusion of Rukh (Chariot in Persian) with Rocca (Italian for fortress).} at each co-ordinate place $e_i:i=1,2,\ldots,s$. Allow each rook to move leftwards not crossing a vertical line, arbitrarily rightwards and similarly downwards not crossing a horizontal line and arbitrarily upwards.

Then pseudo-regularity means that the set of rooks cover every square on this chess-board. In this we may note that, unlike chess, the rooks cover the places at which they stand through the action of $\mathfrak h$. This can fail if we adjoin rooks and the corresponding root vectors are no longer linearly independent. This is the source of the failure of $\mathscr C$ to admit a dense orbit \cite [Lemma 6.10.7]{FJ2}.

Below an example is given when the rooks also cover some elements of $V$. This contradicts \cite [3.1]{FJ2} and arises because $(4)$ is in general a strict inclusion. As we shall see this is caused by VS pairs (\ref {4.3} and \ref {4.4}).

\textbf{Example.} Consider the nilradical $\mathfrak m$ given by the composition $(1,2,2,1)$. Then our procedure \cite [5.4]{FJ2} gives $e=x_{1,2}+x_{2,4}+x_{5,6}$, whilst $V=\mathbb C x_{3,5}+\mathbb C x_{4,6}$. Thus the rooks cover all the squares in $\mathfrak m$ excepting the co-ordinate $(3,5)$ which contradicts \cite [3.1(ii)]{FJ2}. Yet we cannot simultaneously obtain $x_{2,5}$ and $x_{4,6}$ through the action of $x_{4,5}$ since it would have to be used twice. Moreover we recover pseudo-regularity since $x_{4,5}\in V$.

\subsection {Pseudo-regularity} \label {4.2}

\subsubsection {}\label {4.2.1}

\begin {prop} $e$ is pseudo-regular.
\end {prop}

\begin {proof} The proof is by induction on the number of columns. Define $e^\kappa \in \textbf{S}$ as in \ref {3.3.1}. Then we may write $e=e^\kappa+e_k$, with $e_k \in \textbf{C}_k$. Recalling Lemma \ref {3.2.2} we may assume by the induction hypothesis that $e^\kappa$ is pseudo-regular in \textbf{S}.

Thus we only need to show that the rooks coming from $e^\kappa$ and those coming from $e_k$, cover every square in $\textbf{C}_k$.

The proof of the latter is by induction on the height $s$ of $\textbf{C}_k(s)$.

Suppose $s=1$ and let $C_{k-1}$ have height $t$. If (a) of \ref {3.4.1} holds, then the last column is covered by the rook on the $(n-t)^{th}$ row, noting that the downward motion need not cross a horizontal line (though the upward motion may).

For (b), we apply (iv) of \ref {3.4.2} as in the general case below.

By (i) of \ref {3.4} the first $s-1$ columns of $\textbf{C}_k(s+1)$ are those of $\textbf{C}_k(s)$ and so by the induction hypothesis the rooks cover all the first $s-1$ columns of $\textbf{C}_k(s+1)$.

Let us show that this also hold for $\textbf{C}_k(s+1)^s$.

If this column has no $\ast$ then by $(iii)B$ of \ref {3.4}, this column coincides with $\textbf{C}_k(s)^s$ and so again by the induction hypothesis on last column height, the rooks cover all the entries of this column in $\textbf{C}_k(s+1)$.

If this column does have a $\ast$ then by (ii) of \ref {3.4}, it becomes a $1$ in the same place and by (iii) of \ref {3.4} all other changes concern only
$\textbf{C}_k(s+1)^{s+1}$. Thus we are again reduced to the previous case.

It thus remains to show that all the entries of $\textbf{C}_k(s+1)^{s+1}$
are covered. This follows easily from (iv) of \ref {3.4}.

\end {proof}

\subsubsection {A Comment}\label {4.2.2}

The following result is not used; but it indicates what holds for entries in the last column block.

Let $s$ be the height of $C_k$ and set $m= \min {(|S|,s-1)}$.

\begin {lemma} Each of the first $m$ columns of $\textbf{C}_k$ admits a $1$ occurring (of course) on the different rows.
\end {lemma}

\begin {proof} This follows from Lemma \ref {3.6.3}(ii) noting that just a $\ast$ may appear in $\textbf{C}_k(s)^s$. A more direct proof may also be given using \ref {3.4.2}.
%
% The proof is by induction on the height $s+1$ of $C_k(s+1)$. The assertion is empty if $s=0$. By (i) of \ref {3.4.2} we only have to consider the $1$'s adjoined in the last two columns of $\textbf{C}_k(s+1)$ by the algorithm of \ref {3.4.2} and this only for $s+1 \leq m$.
%
%We can assume that $S$ is not empty, otherwise $m=0$ and there is nothing to prove. Moreover by \ref {3.4.1}, for $s=1$, there is either a $1$ or a $\ast$ in the only column of $\textbf{C}_k(1)$.
%
%Now take $m'<m$. Then by the induction hypothesis there is a $1$ in $\textbf{C}_k(m')^{m'}$.
%
%Suppose that there is also a $\ast$ in $\textbf{C}_k(m')^{m'}$. By \ref {3.4.2}(ii), this $\ast$ becomes a $1$ in $\textbf{C}_k(m'+1)^{m'}$, whilst the $1$ becomes a $\ast$ (resp. a $1$) in $\textbf{C}_k(m'+1)^{m'+1}$, if $S$ admits (resp. does not admit) a column of height $m'+1$.
%
%
%Suppose there is no $\ast$ in $\textbf{C}_k(m')^{m'}$. If $t\geq m'+1$, then $S$ meets $R_{m'+1}$, so there is a $\ast$ (resp. a $1$) in $\textbf{C}_k(m'+1)^{m'+1}$, if $S$ admits (resp. does not admit) a column of height $m'+1$.
%% in the $l'+1$ column of $\textbf{C}_k(l'+1)$.
%
%We conclude that the first $m$ columns of $\textbf{C}_k$ admit a $1$ and of course on different rows.
%
%Hence the assertion of the lemma.
\end {proof}

\textbf{Remark}. These $1$'s may be inserted sequentially in the columns going from left to right; but they still may not appear in the lowest possible column. For example consider the composition $(3,2,1,4)$. Then the $1$ in column $9$ occupies row $\textbf{3}$ instead of row $\textbf{4}$.

\subsection {Abstract VS Pairs} \label {4.3}

Here it is convenient to use $k$ as a dummy index for co-ordinates. Yet $k$ was also used as the label of the last column of $\mathscr D$ in \ref {3.2} and already in our previous paper [2.2] \cite {FJ2}. Although this should not cause any confusion, we introduced in \textbf{N.B.} of \ref {3.3.1} the device of replacing the $k$ superscript by $\kappa$ in reference to the last column of $\mathscr D$.

Following the work of Victoria Sevostyanova \cite {Se} a pair of
co-ordinates vectors $x_{i,j},x_{k,l}$ in $\mathfrak m$ is called a VS pair if $x_{j,k}$, lies in $\mathfrak n$ or is a root vector of the Levi factor. Notice here that we must already have $i<j,k<l$.

Given a VS pair, commutation of $x_{i,j}$ with $x_{j,k}$ gives $x_{i,k}$, which lies in the same row as $x_{i,j}$ whilst commutation of $x_{j,k}$ with $x_{k,l}$ gives $x_{j,l}$, which lies in the same column as $x_{k,l}$.

Consider the action of $\mathfrak p$. We cannot simultaneously obtain the entire row to the right of $x_{i,j}$ and the entire column above $x_{k,\ell}$ because the same element of $\mathfrak p$, namely $x_{j,k}$, is being used twice.

In fact this is exactly the source of the strict inclusion in $(4)$. Otherwise the desired regularity of $e$ would result from Proposition \ref {4.2.1}.

In the example of \ref {4.1},
%the pair $x_{2,4},x_{5,6}$ is VS so in fact
both the co-ordinates $(2,5)$ and $(4,6)$ (where an element of $V$ lies) cannot be simultaneously covered. Of course the action of $\mathfrak p$ on $e$ gives a linear combination of these two vectors. Thus a cat zapped by rooks at $(2,4)$ and $(5,6)$ would end up like that of Schroedinger's cat, a linear combination of a live cat and a dead cat. This can get far more complicated and just imagine trying to play chess taking account of VS pairs.

A main result of our present paper is Theorem \ref {4.4.5}, which severely limits the number of VS pairs which need to be considered, those which we shall call ``bad''.

In this we shall not attempt a complete classification of bad pairs for the reasons given in Remark \ref {4.4.3}.

 \subsection {VS Pairs Associated to $e+V$} \label {4.4}

 \subsubsection {VS pairs relative to the canonical Weierstrass section.} \label {4.4.1}

 Recall \ref {3.1} and the notation used there.

 Not all abstract VS pairs are needed.

 \

\textbf{Definition.} A pair $x_{i,j},x_{k,l} \in \supp {e}$, with $j \neq k$, is called a VS pair associated to the Weierstrass section $e+V$ if neither $x_{i,k}$, nor $x_{j,l}$ belong to $\supp{V}$ and in
addition $x_{j,k}$ is a root vector of $\mathfrak p$, called the connecting element of the VS pair.

\

Notice the definition incorporates the fact that by \cite [3.1(ii)]{FJ2}, we cannot expect nor do we need to recover $V$.

\

In particular in the example of \ref {4.1} that we are \textbf{not} considering $x_{2,4},x_{5,6}$ to be a VS pair, since $x_{4,6} \in \supp V$.

\

Recall that $E=\sum_{x_{(i,j)} \in \supp e}\mathbb C x_{i,j}$ and Equation $(1)$.

\subsubsection {VS Quadruplets} \label {4.4.2}

If $C,C'$ are columns of $\mathscr T_\mathfrak m$ , we write $C<C'$ is $C$ lies strictly to the left of $C'$.

\textbf{Definition}. A VS pair $x_{i,j},x_{k,l}$ gives a quadruplet $(i,j,k,l)$, which we call a VS quadruplet if $x_{j,k} \in \supp {V}$.

Given $t \in \{1,2,\ldots,n\}$, let $b(t)$ be the box containing $t$ and $C_{(t)}$ the column containing $b(t)$.

Suppose $(i,j,k,l)$ is a quadruplet, then $C_{(i)}<C_{(j)}\leq C_{(k)}<C_{(l)}$, by definition of $e$.

 Let $c_{(t)}$ denote the height of $C_{(t)}$ for any $t\in \{i,j,k,l\}$.

Note that the condition $C_{(j)}=C_{(k)}$ means that $j,k$ are in the same column and so the connecting element is a root vector lying in the Levi factor $\mathfrak l$ of $\mathfrak p$.

As noted in \ref {4.3} we cannot obtain both $x_{i,k}$ and $x_{j,l}$ as elements of $\mathfrak p.e$ (only a linear combination).
Then for example, the co-ordinate subspace $E_{j,l}$ need not lie in $\mathfrak p.e +V$.

\subsubsection {Bad VS Pairs} \label {4.4.3}

We shall define the notion of a bad VS pair relative to $e+V$ inductively on the number of columns. In this we recall the notation and observation of \ref {3.3.1}.

As \ref {3.3.1}, let $S$ denote the (ordered) set of columns in $\mathscr D$ excluding its rightmost column $C$. We shall sometimes denote $C$ by $C(s)$ to specify its height $s$. We shall sometimes denote $\mathscr D$ by $S\sqcup C$.

 %By Lemma \ref {3.2.2} when $C$ is adjoined to $S$ or when the height of $C$ is changed, the lines joining the columns of $S$ are not changed.

%Let $x_{i,j},x_{k,l}$ be a VS pair defined by $e+V$, with $l \in C$. It can happen that there is either a co-ordinate vector $x_{r,l}:r\neq k$ in $\supp e$ or a co-ordinate vector $x_{i,s}:s \neq j$) in $\supp e$, so that the action of $x_{j,r}\in\mathfrak p$ (resp. $x_{s,k}\in \mathfrak p$) gives $x_{j,l}$ (resp. $x_{i,k}$). In this case either $x_{j,l}$ or $x_{j,l}$ is recovered in a second fashion.

Let $x_{i,j},x_{k,l}$ be a VS pair defined by $e+V$.

\

$(*)$. It can happen that there is either a co-ordinate vector $x_{j,z}\in \supp e$ or a co-ordinate vector $x_{y,k} \in \supp e$, so that the action of $x_{z,l}\in\mathfrak p$ (resp. $x_{i,y}\in \mathfrak p$) gives $x_{j,l}$ (resp. $x_{i,k}$). Consequently either $x_{j,l}$ or $x_{i,k}$ is recovered in a second fashion.

\

%In the first case above we must check that $x_{z,l}$ is not used in a second commutator $[x_{z,l},x_{l,u}]$. If $l \in C$, that is labels an element in the last column, then $x_{l,u}$ cannot lie in $\mathfrak m$, so this holds trivially. In the second case

[\textbf{N.B.}. In the above we also must check that the given element of $\mathfrak p$ is not used in a second commutator $[x_{z,l},x_{l,u}]$ (resp. $[x_{h,i},x_{i,y}]$).

In the first case if $l \in C$, that is labels an element in the last column, then $x_{l,u}$ cannot lie in $\mathfrak m$, so this holds trivially. In the second case we show that the labels of the labels of a quadruplet $(h,i,y,k)$ lie entirely in $S$ and use the induction hypothesis. Specifically this check is carried out in the last part of \ref {4.4.6} for the Levi factor case and following the claim in \ref {4.4.7} for the general case.]

\

When $(*)$ does \textit{not} happen, we say that $x_{i,j},x_{k,l}$ is a bad VS pair defined by $e+V$.

\

\textbf{Remark}. Checking that a pair is not bad can be relatively easy. Checking it is bad is more difficult and this is why we do not attempt to classify all bad pairs.

\begin {lemma} Let $x_{i,j},x_{k,l}$ be a bad VS pair defined by $e+V$. If $l$ does not belong to $C$, it is a bad VS pair relative to $e^\kappa+V^\kappa$.
\end {lemma}

\begin {proof} This is immediate from Lemma \ref {3.2.2}.

%\textbf{N.B.}. In view of this lemma we shall retain the notation with $C$ being the last column of $\mathscr D$ and $l$ being an element of $C$ to the end of Section \ref {4.5}.

\end {proof}

\subsubsection {The Extended Elements} \label {4.4.4}

Following the above lemma we inductively define $e_{VS}$, (resp. $E_{VS}$) by adjoining the co-ordinate vector $x_{j,l}$ (resp. co-ordinate subspace $\mathbb Ce_{j,l}$) to
$e$ (resp. $E$), for each bad VS pair $x_{i,j},x_{k,l}$ as $l$ runs through the columns of $\mathscr D$ going from left to right.

Extending the notation of \ref {3.3.1} we let $e_{VS}^\kappa$ (resp. $E_{VS}^\kappa$) denote the element (resp. subspace $E_{VS}$) obtained as above but in which the contributions coming from the last column $C_k$ are omitted.

%\textbf{N.B.} Even though we augment $e$ we \textit{always} use the original $e$ in the \textit{definition} of a bad VS pair. Simply the adjoined elements which make up $E_{VS}$ can exclude a subsequent VS pair being bad (as for example in the last part of the proof of Lemma \ref {4.4.7}).

\

By definition $E_{VS}$ compensates for the possible strict inclusion in $(4)$, so Proposition \ref {4.2.1}, we obtain

\begin {lemma} $\mathfrak p.e +E_{VS}+V = \mathfrak m$.

\end {lemma}

\textbf{Remark.} We would really like that $\mathfrak p.e_{VS}+V = \mathfrak m$. As we shall see this fails. The basic difficulty is that the roots occurring in $e_{VS}$, as opposed to those occurring in $e$ may not be linearly independent as elements of $\mathfrak h^*$. Thus as noted in \ref {4.1} the rooks defined by $e_{VS}$ need not cover themselves.

\subsubsection {Towards Bad VS Pairs} \label {4.4.5}

 A key result in excluding a VS pair from being bad is the following.

\begin {thm} Let $x_{i,j},x_{k,l}$ be a bad VS pair. Then the connecting element $x_{j,k}$ comes from a line in $\mathscr T_\mathfrak m$ carrying a $*$, equivalently $x_{j,k} \in \supp V$. In particular $(i,j,k,l)$ is a VS quadruplet.

\end {thm}

The proof of the theorem is given below. We retain its hypothesis and notation.

\

In view of lemma \ref {4.4.3} we can adopt the notation of \ref {4.4.3} that $\mathscr D=S\sqcup C$ with $C$ the last column of $\mathscr D$ and $l$ an element of $C$. This notation will be retained until the end of Section \ref {4.5}.

%\subsubsection {A Preliminary Lemma} \label {4.4.6}
%
%The following is proved as part of ``Claim'' in the proof of \cite [Cor. 6.5.4]{FJ2}. For the convenience of the reader we shall repeat the argument.
%
%\begin {lemma} Let $C_1<C_2$ be columns of $\mathscr D$ with boxes $b_i \in C_i:i=1,2$. Let $R_t$ be the row in which $b_2$ occurs. Assume that $\ell_{b_1,b_2}\in \ell(\mathscr D)$. Then there exists a column $C': C_1\leq C' <C_2$ of height at least t -1.
%\end {lemma}
%
%\begin {proof} This is immediate unless $b_1 \in R^{t-2}$.
%
%If not, $\ell$ must be left and downward going by at least two rows. This can only happen if there is a line $\ell'$ carrying a $\ast$, gated at the $t^{th}$ stage, with the same left hand end-point $b_1$ as $\ell$, so allowing $\ell$ to be drawn \cite [5.4.7]{FJ2}. By up-going linkage \cite [5.4.3]{FJ2} the right hand end point $b'$ of $\ell'$ lies in $R_{t-1}$ and since $\ell'$ carries a $\ast$, the box $b'$ lies at the bottom of the column $C'$ in which it is contained, so $C'$ has height $t-1$. Yet $C_2$ has height $\geq t$ so by adjacency \cite [5.4.3]{FJ2}, $C'$ lies strictly to the left of $C_2$ and of course to the right of $C_1$.
%
%\end {proof}
%

\subsubsection {The Levi Factor Case} \label {4.4.6}

Suppose that in the quadruplet $(i,j,k,l)$, the boxes $b(j),b(k)$ lie in the same column, that is $C_{(j)}=C_{(k)}$. (This means that $x_{j,k}$ belongs to the Levi factor of $\mathfrak p$.)

%In view of lemma \ref {4.4.3} we can assume that $C$ is the last column of $\mathscr D$ and $l$ an element of $C$.
%Assume this column has height $t$.

\begin {lemma} If $C_{(j)}=C_{(k)}$, then $x_{i,j},x_{k,l}$ is not a bad VS pair. In other words the assertion of the theorem holds in this case.
\end {lemma}

\begin {proof}

Let $t$ be the height of $C_{(j)}=C_{(k)}$.

Suppose first that $j<k$. By hypothesis $b(k)$ has a right going line $\ell$ to $b(l)$ labelled by a $1$. Then by \cite [5.4.8(v)]{FJ2}, the height of $C_{(l)}$ is at most one less than the height of $C_{(k)}$. Yet $b(j)$ lies above $b(k)$ in $C_{(j)}=C_{(k)}$,
so $b(j)$ is not right extremal. Thus there exists
 a right going line $\ell'$ from $b(j)$.

Let $b(z)$ be the right end-point of $\ell'$.
Suppose $\ell'$ is labelled by a $\ast$, then by \cite [6.5.4]{FJ2}, one has $b(z)=b(\ell)$. Thus the line $j,l$ is labelled by a $\ast$ and so the pair $(i,j),(k,l)$ is not a VS pair (Definition \ref {4.4.2}).

 Thus we can suppose that $\ell'$ is labelled by a $1$, in other words $(j,z) \in \supp {e}$. Moreover either $z\in C$ and $x_{z,l}$ belongs to the Levi factor or $z<l$. In both cases $x_{z,l}$ is a root vector of $\mathfrak p$ and we obtain $x_{j,l}$ in a second fashion as $[x_{j,z},x_{z,l}]$. Again $x_{z,l}$ cannot be used in a commutator $[x_{z,l},x_{l,u}]$ because $l$ labels a column in the last column block $\textbf{C}$ and so $x_{l,u}$ cannot lie in $\mathfrak m$.

Suppose now that $j>k$. By hypothesis $b(j)$ has a left going line $\ell$ to $b(i)$, labelled by a $1$ and $b(k)$ lies above $b(j)$ in $C_{(j)}=C_{(k)}$. Then by \cite [6.5.4] {FJ2}, there exists a left going line $\ell'$ from $b(k)$ labelled by a $1$. Let $b(y)$ be the left-hand end-point of $\ell'$. Since $y$ occurs in a column strictly to the left of $C_{(j)}$, necessarily $y<k$.

Let $R_t$ be the row in which $b(j)$ occurs. Take $b_1=b(i),b_2=b(j)$ in Lemma \ref {3.6.1}. By its conclusion there is a column $C'$ of height $\geq t-1$ to the right of $C(i)$ and strictly to the left of $C(j)$. Since the only line carrying a $\ast$ having a box in $C_{(j)}$ as its right hand end point must lie strictly below $b(j)$, every left-going line from a box in $C_{(j)}$ above $b(j)$ must be downward going by \cite [5.4.8(ix)]{FJ2} and hence meet $C'$ or a column between $C',C_{(j)}$. In both cases this column lies to the right of $C_{(i)}$, though not necessarily strictly.

Thus either $i<y$ or both $b(y),b(i)$ lie in $C_{(i)}$ and in this case $x_{i,y}$ lies in the Levi factor. In both cases we recover $x_{i,k}$ as the commutator $[x_{i,y},x_{y,k}]$.

Finally suppose $x_{i,y}$ is used in a second commutator $[x_{h,i},x_{i,y}]=x_{h,y}$, coming from a second VS pair $x_{h,i},x_{y,u}$. Then the line $\ell_{y,u}$ carries a $1$. However the line $\ell_{y,k}$ is also labelled by a $1$ and by \cite [5.4.8(vi)]{FJ2} this forces $u=k$. Then the resulting quadruplet $(h,i,y,k)$ lies entirely in $S$.

Thus if this second VS pair is bad we can assume by the inductive construction \ref {4.4.4} that the subspace $\mathbb Cx_{i,k}$ has already been adjoined to $E^\kappa_{VS}$ and thus we recover $x_{h,y}$, as required.

 \end{proof}

\subsubsection {The General Case.}\label {4.4.7}

The proof of general case is similar, but one has to be more careful in view of the possible appearance of $\ast$ which allows the VS pair to be a VS quadruplet.

Let $(i,j),(k,l)$ be a VS pair, with $l \in C$, which does \textbf{not} form a VS quadruplet and that $b(j),b(k)$ are in distinct columns, so then $i<j<k<l$. By definition of a VS pair, $b(j)$ (resp. $b(k)$) admits a left (resp. right) going line to $b(i)$ (resp. $b(l)$). Under these conditions one has the

\

\textbf{Claim}. Either $b(j)$ admits a right going line labelled by a $1$, to $b(u):j<u<l$, or $b(k)$ admits a left going line labelled by a $1$, to $b(v):i<v<k$.

\

In the above, the inequality $j<u$ (resp $v<k$) follows because the box joined to $b(j)$ (resp. $b(k)$) lies strictly to the right (resp. left) of $C_{(j)}$ (resp. $C_{(k)}$). The second pair of inequalities is more delicate.

\

Assume the claim holds.

\

In the first case we obtain $x_{j,l}$ as $[x_{j,u},x_{u,l}]$ in the second case we obtain $x_{i,k}$ as $[x_{i,v},x_{v,k}]$.

As in the case when $b(j),b(k)$ are in the same column, we cannot use $x_{u,l}$ again with respect to a further VS pair $x_{j,u},x_{l, t}$ because $l\in C$ and so $x_{l,t} \notin \mathfrak m$.

As in the case when $b(j),b(k)$ are in the same column, one may use $x_{i,v}$ for a second VS pair $x_{f,i},x_{v,k}$; but this will lie entirely in $S$. As in the proof of Lemma \ref {4.4.6}, if this pair is bad, we can assume that the subspace $\mathbb Cx_{i,k}$ has already been adjoined to $E^\kappa_{VS}$.
.

Thus the theorem results from the claim, which we now proceed to prove.

\subsubsection {Proof of claim.}\label {4.4.8}

 Suppose first that $t:=c_{(j)}=c_{(k)}$. Then by \cite [5.4.8(v)]{FJ2} every box in $R^t$ lying in $C_{(j)}$ (resp. $C_{(k)}$) has a right (resp. left) going line to a box between these two columns. In particular this box may be $b(j)$ (resp. $b(k)$) and writing the other end point of this line as $b(u)$ (resp. $b(v)$) we obtain the remaining inequality $u<l$ (resp. $i<v$). One of these two lines carry a $1$ with the following possible exception. Namely, the right going line from $b(j)$ is the same as the left going line from $b(k)$ and this line carries a $\ast$. However this means that $(i,j,k,l)$ is a VS quadruplet.

 Thus we can assume $c_{(j)}\neq c_{(k)}$.

 Suppose that $c_{(j)}> c_{(k)}$ and Set $t =c_{(j)}$. Then $C_{(j)}$ admits a nearest column $C'$ of height $t$ strictly to the right of $C_{(k)}$ or we can adjoin one to the extreme right of $\mathscr D$. By Lemma \ref {3.2.2}, this adjunction will not change the left going lines from boxes in $C_{(k)}$.

 Let $C''$ be a left neighbour to $C'$. It lies to the right of $C_{(j)}$ and strictly to the left of $C_{(k)}$.

 By \cite [5.4.8(v)]{FJ2} every box in $C_{(k)}$ (so particularly $b(k)$) has a left going line $\ell$ to a box lying to the right of $C''$ labelled by a $1$, unless $\ell$ goes leftwards from the box $C'\cap R_t$. The latter is impossible because $C'$ lies strictly to the right of $C_{(k)}$. Finally the box $b(v)$ joined to $b(k)$ by $\ell$ lies to the right of $C''$, so to the right of $C_{(j)}$, so strictly to the right of $C_{(i)}$. This gives the required inequality $i<v$.

 It remains to consider the case $c_{(j)}<c_{(k)}$, which is similar but a little more complicated.

 Set $t=c_{(k)}$. Then $C_{(k)}$ admits a column $C''$ of height $t$ strictly to the left but nearest to $C_{(j)}$ or we can adjoin it as a left neighbour to $C_{(k)}$ to the extreme left of $\mathscr D$. By Lemma \ref {3.2.3}, and recalling that $c_{(j)}< c_{(k)}$, this adjunction will not change the right going lines from boxes in $C_{(j)}$, though it may change their labels.

 % By \cite [5.4.8(v)]{FJ2} every box (so in particular $b(j)$) in $C_{(j)}$ has a right going line $\ell$. Moreover $\ell$ is labelled by a $1$, unless it goes rightwards from the box $C'\cap R_t$, which is impossible because $C'$ lies strictly to the left of $C_{(j)}$. Finally the box $b(u)$ joined to $b(j)$ by $\ell$, lies to the left of $C_{(k)}$ since $C_{(k)}$ has height $t$, so strictly to the left of $C_{(l)}$. This gives the required inequality $u<l$.

 By \cite [5.4.8(v)]{FJ2} the box $b(j)$ in $C_{(j)}$ has a right going line $\ell$ to a box at the left of $C_{(k)}$, hence strictly to the left of $C_{(l)}$. If it is labelled by a $1$, the claim holds in this case as in the previous case.

 Otherwise by \cite [5.4.8(v)]{FJ2}, there is exactly a right going line $\ell'$ from $b(j)$ labelled by a $\ast$ whose right end point is the box $b'$ in $C_{(k)}\cap R_t$. Notice that this also means that $C_{(k)}$ must already have a left neighbour $C''$ in $\mathscr D$, so adjoining a left neighbour is unnecessary. Moreover by adjacency \cite [5.4.3,5.4,11]{FJ2}, $\ell'$ cannot cross a column of height $\geq t$, so there is no column of height $t$ strictly between $C_{(j)},C_{(k)}$. Thus $C''$ is a left neighbour to $C_{(k)}$ in $\mathscr D$.

 If $b'=b(k)$, then $(i,j,k,l)$ is a VS quadruplet and we are done.

 However by definition $b(k) \in C_{(k)}$ and so if $b(k)\neq b'$, then $b(k)$ must appear strictly above $b'$, that is as $C_{(k)}\cap R_{t'}$ for some $t'<t$.

Yet by Lemma \ref {3.6.1} there is a column $C'$ of height $\geq t-1$ with $C_{(j)} \leq C' <C_{(k)}$. Then exactly as in the proof of the last part of Lemma \ref {4.4.7} we conclude by \cite [5.4.8(v)]{FJ2} that $b(k)$ admits a left going line $\ell''$ labelled with a $1$ to some box $b(v)$ to the right of $C'$. Yet $C'$ lies to the right of $C_{(j)}$ and thus strictly to the right of $C_{(i)}$. Consequently $i<v$, as required.

% If $C'$ lies to the right of $C_{i)}$, then either $i<v$ or $x_{i,v}$ ;lies in the Levi factor and we are done.
%
% It remains to consider the case when $C'$ lies strictly to the left of $C_{(i)}$.

 %Finally the box $b(u)$ joined to $b(j)$ by $\ell$, lies to the left of $C_{(k)}$ since $C_{(k)}$ has height $t$, so strictly to the left of $C_{(l)}$. This gives the required inequality $u<l$.

 This completes the proof of the claim and taking account of Lemma \ref {4.4.7} the theorem.

 \

 \textbf{Remark.} This result cuts down drastically the number of bad VS pairs. Indeed by \cite [5.4.8(vi)]{FJ2} there is at most left (resp) right going line carrying a $1$ from a box in $\mathscr T_\mathfrak m$. We conclude that a line, say $\ell_{j,k}$, carrying a $\ast$ uniquely determines a VS quadruplet $(i,j,k,l)$ (if one exists). Thus the number of bad VS pairs is at most $g$.

 \

 \textbf{Examples.}

 \

 $(i)$. Consider the array $(1,2,1,1,2,2,3)$. Take $i=3,j=4,k=9,l =11$. Then $(3,4),(9,11)$ is a VS pair. Now the lines $\ell_{4,7}$ and $\ell_{7,9}$ are the obvious choices for the lines of the claim but both carry a $\ast$, so at first sight one seems to get a contradiction. However there is also a line $\ell_{4,9}$ carrying a $1$ obtained by the joining of loose ends \cite [5.4.7(ii)]{FJ2}. Such joinings are essential to validate \cite [5.4.8(v)]{FJ2} used in the above proof.

 \

 $(ii)$. Consider the array $(1,2,1,1,2,3)$. Then the VS pair $(3,4),(7,9)$ forms a VS quadruplet. Actually since the line $(4,10)$ carries a $1$ and $10$ lies in the same column as $9$, $x_{4,9}$ can be recovered in a second fashion as $[x_{4,10},x_{10,9}]$. Thus this VS pair is not bad.

 \

 $(iii)$. Consider the array $(3,2,1,1,2,3)$. There are just two VS quadruplets $(4,6,7.8)$ and $(4,6,9,11)$. One may easily check that the first is not bad. Thus the second being not bad would imply that $e$ is regular. One can check that the second is indeed bad, but a better method is to show that $e$ cannot be regular by computing its nilpotency class as in \cite [6.10.9, Example 1]{FJ2}.

 Again as in loc. cit. one may adjoin the co-ordinate vector $(6,11)$ to $e$. The resulting element $e_{VS}$ \textit{is} regular because the graph defined by the roots of the root vectors in $e_{VS}$ has no cycles and so these roots are linearly independent. More details of this example are in section \ref{6} figure \ref{fig3}.

\subsection {Weierstrass Sections} \label {4.5}

\subsubsection {A Preliminary lemma} \label {4.5.1}

Recall the subalgebra $\mathfrak u_{\pi'}$ of $\mathfrak m$ defined in \cite [2.5]{FJ2}. It is spanned by root subspaces and those which do not appear in $\mathfrak m$ are called the excluded root vectors \cite [2.7]{FJ2}. The set of all such excluded root vectors in $\mathfrak m$is denoted by $X$ in \cite [6.9.5]{FJ2}. In the matrix presentation, that is in \textbf{M}, they are encircled by $O$ \cite [6.9.5]{FJ2}. (Here we shall not use this diagrammatic presentation but a reader working through examples is urged to do so.)

We construct $E_{VS}$ by adjoining $\mathbb Cx_{j,l}$ for each bad VS pair $x_{i,j},x_{k,l}$. In this Theorem \ref {4.4.5} allows us to assume that $(i,j,k,l)$ is a VS quadruplet.

We show that $E_{VS} \subset \mathfrak u_{\pi'}$. To do this we shall \textit{not} need to require that $(i,j,k,l)$ is bad, as a VS quadruplet.

 Further to the definition of $X$ above, we recall \cite [6.9.3]{FJ2} that every $\ast$ is encircled. Following \cite [6.9.5]{FJ2} we let that $Y$ be the set of co-ordinate vectors with a $O$ encircling a $\ast$ and $Z$ the set of co-ordinate vectors which are just encircled by an $O$.

 Finally set $Z=X\setminus Y$.

 %These co-ordinate vectors are entries of \textbf{M}.

 The hypothesis of the following lemma holds if for example $(i,j,k,l)$ is a VS quadruplet.

\begin {lemma} Suppose $x_{j,k}\in Y$. If $x_{k,l}$ is a co-ordinate vector of $e$, for some $l>k$, then $x_{j,l} \notin X$. In other words $x_{j,k}$ is encircled but its ``neighbour'' $x_{j,l}$ is not.
\end {lemma}

\begin {proof} The set $X$ is the union (not necessarily disjoint) of the sets $X_{C',C}$ of elements excluded by a pair of neighbouring columns $C',C$ of height, say $s$. The elements of a given $X_{C',C}$ have co-ordinates from the entries of the boxes in the columns between $C',C$ and moreover are determined by the distribution of these entries through a simple combinatorial procedure \cite [2.7]{FJ2}.

 Again by Lemma \ref {3.2.2}, the hypothesis of the lemma does not depend on the columns strictly to the right of the column containing $l$. Thus we may assume without loss of generality that $C$ is the rightmost column of $\mathscr D$.

By \cite [2.7]{FJ2}, the subset $X_{C',C}$ of elements of $X$ obtained from the neighbouring pair $C',C$ is described in the following fashion.

Let $C_1,C_2,\ldots,C_u$ be the columns of height strictly greater than $s$ between $C',C$. For all $i=1,2,\ldots,u$, let $c_i$ denote the height of $C_i$.

Recall the notation of \ref {2.2}. After \cite [2.7]{FJ2} we have the following result.

 The excluded root vectors in \textbf{M} from set $X_{C',C}$ lie in the rectangles $R_{j'}:j' \in [1,u+1]$ formed from the last $c_{j'}-s:j'\leq u$ columns (or last column if $j'=u+1$) of $\textbf{C}_{j'}$ starting from the highest row of the block $B_{{j'}-1}$ to the row just above the highest row of $B_{j'}$.

 Moreover each such rectangle $R_{j'}:j' \in [1,u+1]$ has a gap between the $(s+1)^{th}$ and $c_{j'-1}^{th}$ row relative to $B_{j'-1}$, where exactly no excluded roots lie.

 For $(3)$ below, one should keep in mind that the origin of this gap is in the columns of height $<s$ lying strictly between $C_{j'-1}$ and $C_{j'}$.

Furthermore

\

\textbf{Remark 1.} Recall that by \cite [6.6.2]{FJ2} an element of $Y$ in a column block always lies in its last column.

\

\textbf{Remark 2.} No co-ordinate of $e$ is excluded, equivalently no $1$ in $\textbf{M}$ is encircled \cite [Cor. 6.9.3]{FJ2}.

\

There are three cases to consider.

In the first and second of these, the gap noted above will cause no problems, in the last case it might.

\

$Case (1)$. The element $x_{j,k} \in Y$, lies in a column block of a column of height $>s$, that is some $\textbf{C}_{j'}:j' \in [1,u]$ in the above notation.

\

By Remark $1$, $x_{j,k}$ lies in the last column of $\textbf{C}_{j'}$. Thus an entry $x_{j,l}:l>k$ must lie in one of the column blocks $\textbf{C}_{j''}:j''>j'$ and in particular strictly above the rows of the $B_{j'}$. Thus it cannot lie in any of the rectangles $R_{j''}:j'' \in [1,u]$, for $j''>j'$ and nor of course for $j''\leq j'$. Thus $x_{j,l}$ is not an excluded co-ordinate and so not encircled.

\

$Case (2)$. The element of $x_{j,k}\in Y$ lies in a column block of a column of height $s$, that is to say from $\textbf{C}_k$.

\

In this case the co-ordinate vector $x_{j,l}$ strictly to right of $x_{j,k}$ does not even exist.

\

$Case (3)$. The element of $x_{j,k} \in Y$ lies in a column block of a column of height $s'<s$.

\

This means that $x_{j,k}$ belongs to a column $C''$ of height $<s$, so lying strictly between $C_{j'-1}$ and $C_{j'}$, for some $j' \in [2,u]$.

Recall that the element $x_{j,k} \in Y$ given by \cite [Sect. 5.4.8]{FJ2} obtains by a two-fold process. First in step $2$ of \cite [5.3]{FJ2}, one starts from a horizontal line carrying a $\ast$ going rightwards to the box in $R_{s'}\cap C''$. Then at Step $3$, this line may have its left hand end point moved down to the top of a column $C'''$ of even lower height than $C''$ \cite [5.4.7,5.4.8]{FJ2}, but without crossing a column of height by $\geq s$ via adjacency \cite [Lemma 5.4.8 (vii)]{FJ2}.

We conclude that $C'''$ also lies strictly between $C_{j'-1}$ and $C_{j'}$.

 As an immediate consequence

\

$(\alpha)$. The row containing $x_{j,k}$ meets $R_{j'}$ strictly below its gap.

\

Yet by the hypothesis the line $\ell_{k,l}$ carries a $1$ and so by Remark 2, the root vector $x_{k,l}$ is not excluded, equivalently $x_{k,l}$ is not in $X$. Then

\

 either
no co-ordinate vector $x_{k',l}:k'\leq k$ is excluded and in particular $x_{j,l}\notin X$,

\

or $x_{j,l}\in X$, \textit{and}
$x_{k,l}$ lies in a gap of some $R_{j''}:j''\in [1,u]$.

\

$(\beta)$. In the second case, $x_{j,l}$ lies in that part of $R_{j''}$ strictly above its gap.

\

Since the rows of distinct rectangles are disjoint, $(
\alpha),(\beta)$ force $j'=j''$.

Then $(\alpha),(\beta)$ become contradictory excluding the second option.

Hence the conclusion of the lemma.

%In cases (2) and the last part of (3) we obtained a contradiction, in the remaining case $x_{j,l}$ was not an excluded co-ordinate. Hence the conclusion of the lemma.

\end {proof}

\subsubsection {VS Quadruplets} \label {4.5.2}

Let $(i,j,k,l)$ be a VS quadruplet. Recall the definition of $E_{VS}$ given in \ref {4.4.4}.

\begin {cor} There is no element of $Z$ at the co-ordinate $j,l$. In particular $E_{VS} \subset \mathfrak u_{\pi'}$.
\end {cor}

\begin {proof}

Since by definition of a VS quadruplet, an element of $Y$ occurs at the co-ordinate $(j,k)$, whilst $1$ occurs at $(k,l)$. Then the hypothesis of Lemma \ref {4.5.1} and by its conclusion there is no element of $Z$ at the co-ordinate $(j,l)$, as required.

The second part is immediate.

\end {proof}

\textbf{Remark.} Return to the parabolic defined by the composition $(3,2,1,1,2,3)$, that is Example (iii) of \ref {4.4.8}. In this the VS quadruplet $(4,6,9,11)$ is bad. According to the corollary, $x_{6,11} \notin X$, as may be checked. Thus this root subspace may be included in $E_{VS}$ to retain the property that $E_{VS} \subset \mathfrak u_{\pi'}$. Yet $x_{4,9} \in X$, so this property would be destroyed had we instead required this second root subspace to lie in $E_{VS}$.

This is why we chose to adjoin the ``right-hand'' element $x_{j,l}$ to $e_{VS}$ and not the ``left-hand'' element $x_{i,k}$. It is remarkable that this choice always works.

\subsubsection {VS Quadruplets} \label {4.5.3}

Recall the definition of $\mathscr N$ given in \ref {1.3}. The following key result was announced in \cite [Prop. 6.10.]{FJ2}. Here it is proved/

\begin {prop}

\

(i). $P.E_{VS} \subset \overline{B.\mathfrak u_{\pi'}}:=\mathscr N^e$.

\

(ii). $\dim P.E_{VS} =\dim \mathfrak m - g$.

\

(iii). $\overline{P.E_{VS}} =\overline {B.\mathfrak u_{\pi'}}$.
\end {prop}

\begin {proof} By \cite [Cor. 6.9.8]{FJ2}, $\mathscr C$ is a component of $\mathscr N$, hence $P$ stable. Then (i) follows from Corollary \ref {4.5.2}.

Let $F \subset E_{VS}$ be a vector space complement to $\mathfrak p.e$ in $\mathfrak p.e+E_{VS}:=U$. So then
$$U=\mathfrak p.e \oplus F. \eqno{(6)}$$

Now $e,F$ both belong to the vector space $E_{VS}$, so $P.(e+F) \subset P.E_{VS}\subset \mathscr C$, by (i).

%We claim that $P.(e+F)$ is dense in $U:=\mathfrak p.e +E_{VS}=\mathfrak p.e +F$.

Following the (standard) argument given on p. 262 of \cite [Lemma 8.1.1]{D}, let $(y_1,\ldots,y_r)$ be a basis for $F$. Then we have a morphism $\psi:P\times F \mapsto \mathscr C$ defined by $\psi(p,\xi_1,\ldots,\xi_r)=p.(e+\sum_{i=1}^r\xi_iy_i)$.

Through $(6)$ it follows as in \cite [Lemma 8.1.1]{D} that image of the tangent map at $0 \in \mathbb C^r$ is $U$. Thus $\dim P.E_{VS} \geq \dim P.(e+F) =\dim U$.

On the other hand $U+V= \mathfrak m$, by Lemma \ref {4.4.4}, and so $\dim U \geq \dim \mathfrak m - \dim V$. Thus
$$\dim P.E_{VS}\geq \mathfrak m - \dim V. \eqno {(7)}$$

Yet $\dim \overline{B.\mathfrak u_{\pi'}}= \dim \mathfrak m-\dim V$, by \cite [Thm. 6.9.7]{FJ2}.

 Combined with (i), this gives the opposite inequality to $(7)$. Hence (ii).
%
%On the other hand $e \in E_{VS}$, so $\overline{P.E_{VS}}=\overline{P.(e+E_{VS})}\supset\overline{P.(e+F)}=U$, by the claim. %Yet by Proposition \ref {4.2}, and it follows that $P.E_{VS}$ is dense in $U$y, which
%Yet by Lemma \ref {4.4.4}, $U$ has dimension $\geq \dim \mathfrak m - \dim V$.
%
%Combined with $(6)$, this gives (ii).
%
% %By Lemma \ref {4.4.4} one has $\dim \mathfrak p.E_{VS} \geq \codim V=\dim \mathfrak m - g$. By Corollary \ref {4.4.10} and \cite [Thm. 6.9.7]{FJ2}, $\dim \mathfrak p.E_{VS} \leq \dim B.\mathfrak u_{\pi'}=\dim \mathfrak m - g$. Hence (ii).
%
% %Note also that equality in $(6)$ gives (iii).

 Finally (iii) follows from (i),(ii) and the irreducibility of the right hand side.

 \end {proof}

 \textbf{Remarks.} This proves \cite [Prop. 6.10.4]{FJ2}.
 %As noted there it follows that $\overline {P.E_{VS}}$ is the canonical component $\mathscr C= \overline {B.\mathfrak u_{\pi'}}$ of $\mathscr N$.
 On the other hand by \cite [6.10.7]{FJ2} it cannot be that $P.e_{VS}$ is always dense in $\mathscr C$. Again as we saw in \ref {4.1}, even the dimension of $\mathfrak p.E$ can strictly exceed $\dim \mathfrak m - g$.

\subsubsection {Weierstrass Sections} \label {4.5.4}

Fix a standard parabolic subalgebra $\mathfrak p$ of $\mathfrak g$, with $\mathfrak m$ its nilradical and $\mathscr T_\mathfrak m$ the tableau it defines.

Let $C,C'$ be a pair of neighbouring columns of $\mathscr T_\mathfrak m$ having height $s$ and $f_{C,C'}$ the corresponding Benlolo-Sanderson invariant and let $\varphi:(\mathbb C [\mathfrak m])^{P'} \rightarrow \mathbb C[e+V]$ (resp. $\varphi_{VS}:(\mathbb C [\mathfrak m])^{P'} \rightarrow \mathbb C[e_{VS}+V]$) be defined by restriction of functions.

\begin {prop}

\

$(i)$. For every pair of neighbouring columns $C,C'$ of $\mathscr T_\mathfrak m$, one has $\varphi(f_{C,C'})=\varphi_{VS}(f_{C,C'})$.

\

$(ii)$. $e_{VS} + V$ is a Weierstrass section for the action of $P'$ on $\mathfrak m$.
\end {prop}

\begin {proof}

Recall \cite [4.2.5]{FJ1} that every monomial in $f_{C,C'}$ corresponds to a disjoint union of composite lines passing though \textit{every} element in $R^s$ between $C,C'$. We shall refer to this last property by saying the disjoint union is complete.

Recall that our construction in \cite [5.4.11]{FJ2} gives one particular complete disjoint union of composite lines, formed from the lines defining $e+V$.

Moreover in loc. cit. we also showed that there are no other possible complete disjoint union formed from the lines defining $e+V$.

\

\textbf{Observation}. To prove (i) it is obviously enough to show that the additional lines given by adjoining the lines $\ell_{j,l}$ for every VS quadruplet $(i,j,k,l)$ does not introduce a new complete disjoint union of composite lines.

\

Recall (Remark \ref {4.4.8}) that a line, say $\ell_{j,k}$, carrying a $\ast$ uniquely determines a VS quadruplet $(i,j,k,l)$ (if one exists). Moreover the line $\ell_{j,l}$ introduced by this quadruplet is uniquely the composition of the lines $\ell_{j,k},\ell_{k,l}$ \textit{and so} uniquely determined by $\ell_{j,k}$.

\

$(*)$. Thus there can be only new composite line $\mathscr L'$ which simply bypasses the box with entry $k$. Moreover the value it gives to the monomial, changes from $x_{j,k}$ to $1$.

\

Yet as we shall see below even this extra scalar factor cannot appear!

\

Recall that in \cite [5.4.11]{FJ2} that we proved that there is only one complete disjoint union of composite lines joining $C,C'$, formed from the lines defining $e+V$. As noted in loc cit, this was practically obvious if we forbid passage across gated lines (carrying a $\ast$), Now in this the only line, call it $\ell_{j,k}$, carrying a $\ast$ which is not gated, is exactly the unique line whose right hand end point is the box $C'\cap R_s$.

Moreover going leftwards from this box there is a unique composite line $\mathscr L$ reaching a box in $C$. The boxes through which $\mathscr L$ passes cannot, by definition, appear in the passage of a further composite line $\mathscr L'$ disjoint to $\mathscr L$. As note in \cite [5.4.11]{FJ2} this forbids passage across \textit{all} the lines carrying a $\ast$ gated at the $s-1$ stage. Then by downward induction (cf \cite [5.4.11]{FJ2}, one concludes that there is only one complete disjoint union of composite lines joining $C,C'$. Moreover amongst the set of individual lines in these composite lines, there is only one carrying a $\ast$ namely $\ell_{j,k}$.

We conclude that the (unique) complete disjoint union of composite lines determined by $e+V$ is unchanged, when we pass to $e_{VS}$ because no line carrying a $\ast$ is used in the first union except $\ell_{j,k}$ and this cannot be replaced by $\ell_{j,l}$ because $k$ is an entry of $C'$ and therefore $l$ is in a box strictly to the right of $C'$. Thus the potentially new composite line $\mathscr L'$ mentioned in $(*)$ above does not even exist. Hence (i).

(ii) follows from (i) and \cite [5.4.12]{FJ2}.

\end {proof}

\section{Figures}\label{6}

The following figures illustrate the assertions of \ref {3.4}, for different assignments of column heights of $S$. In each case three pictures are drawn. On the left (resp. right) of the top picture, we give the last (resp. last two) columns of $\textbf{M}(s)$ (resp. $\textbf{M}(s+1)$). In the middle picture, a composition realizing these assignments of height is given. In this $s$ is its last entry. The bottom picture describes the resulting matrix presentation of $\textbf{M}(s) \rightarrow \textbf{M}(s+1)$.

%Figures $1,2,3$ illustrates a fixed choice of $S$ with $s$ increasing by one. %OMIT THE REMAINING FIGURES.

\subsection {Case $1$}\label{6.1}

\

\

\begin{center} \tiny
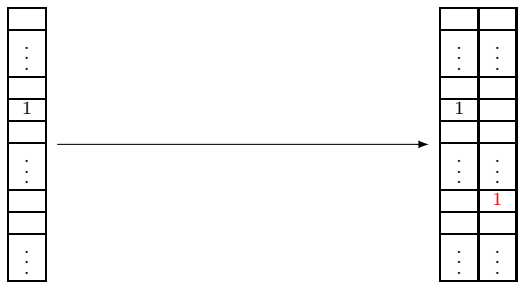
\begin{figure}[hbt!]
\begin{tikzpicture}[grow=right,
sloped,
edge from parent/.style={draw,-latex}, level distance=6cm, sibling distance=5cm]

\node{ $\left.\begin{array}{|c|}\hline \\\hline \vdots \\\hline \\\hline {1} \\\hline \\\hline \vdots \\\hline \\\hline \\\hline \vdots \\\hline \end{array}\right.$
}
child {
 node{ $\left.\begin{array}{|c|c|}\hline & \\\hline \vdots& \vdots \\\hline & \\\hline {1}& \\\hline & \\\hline \vdots&\vdots \\\hline & \red{1} \\\hline & \\\hline \vdots & \vdots \\\hline \end{array}\right.$}
edge from parent
%node[above] { no column of height $s+1$ in $\bold{S}$}
node[below] { }};
\end{tikzpicture}\\
\caption{$S$ has a column of height $>s+2$ but none of height $s,s+1$. This illustrates the last part of $(iii)B$ }
\end{figure}
\end{center}
\textbf{Example:} Consider the parabolic defined by the composition $(3,1)$:
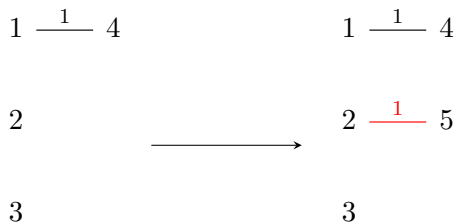
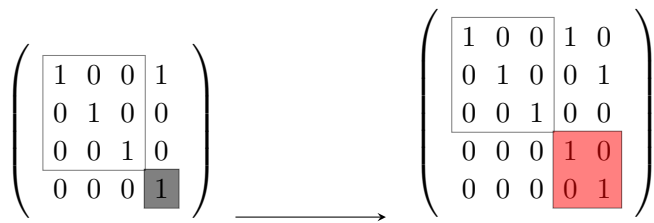
\begin{center}

\begin{tikzcd}[row sep=1.8 em,%tiny,
column sep = 2 em]
1\arrow[-,r,"1"]&4\\
2&\\
3&\\
\end{tikzcd}
\;\begin{tikzpicture}

\draw [->,>=stealth] (0,.5) -- (2,.5);
\end{tikzpicture} \;
\begin{tikzcd}[row sep=1.8 em,%tiny,
column sep = 2 em]
1\arrow[-,r,"1"]&4\\
2\arrow[-,r,"1", red]&5\\
3\\
\end{tikzcd}

\end{center}

\begin{center}
 \begin{tikzpicture}
 \matrix [matrix of math nodes,left delimiter=(,right delimiter=)] (m)
{
1 & 0 & 0& 1 \\
0 & 1 & 0& 0 \\
0 & 0 & 1& 0 \\
0 & 0 & 0& 1 \\
};

\draw[
opacity=0.5] (m-3-1.south west) rectangle (m-1-3.north east);
\draw[
fill=, opacity=0.5] (m-4-4.south west) rectangle (m-4-4.north east);

\end{tikzpicture}
\;\begin{tikzpicture}

\draw [->,>=stealth] (0,.5) -- (2,.5);
\end{tikzpicture} \;
 \begin{tikzpicture}
 \matrix [matrix of math nodes,left delimiter=(,right delimiter=)] (m)
{
1 & 0 & 0& 1 &0 \\
0 & 1 & 0& 0 &1 \\
0 & 0 & 1& 0 &0 \\
0 & 0 & 0& 1 &0 \\
0 & 0 & 0& 0 &1\\
};

\draw[
opacity=0.5] (m-3-1.south west) rectangle (m-1-3.north east);
\draw[
fill=red, opacity=0.5] (m-5-4.south west) rectangle (m-4-5.north east);
\end{tikzpicture}
\end{center}
\

\

\subsection {Case $2$}\label{6.2}

\begin{figure}[H]
%\begin{center}
\centering \tiny
\begin{tikzpicture}[grow=right,
sloped,
edge from parent/.style={draw,-latex}, level distance=6cm, sibling distance=5cm]

\node{ $\left.\begin{array}{|c|}\hline \\\hline \vdots \\\hline \\\hline {1} \\\hline \\\hline \vdots \\\hline \\\hline \vdots \\\hline \\\hline \vdots \\\hline \end{array}\right.$
}
child {
 node{ $\left.\begin{array}{|c|c|}\hline & \\\hline \vdots& \vdots \\\hline & \\\hline {1}& \\\hline & \\\hline \vdots&\vdots \\\hline & \red{\ast} \\\hline \vdots & \vdots \\\hline & \red{1} \\\hline \vdots&\vdots \\\hline \end{array}\right.$}
edge from parent
node[above] { }
node[below] { }};
\end{tikzpicture}\\
\caption{ $S$ has no column of height $s$, a column of height $s+1$ and a column of height $>s+1$. This illustrates the first part of $(iii)B$ and $(iv)$.}
%\end{center}
\end{figure}
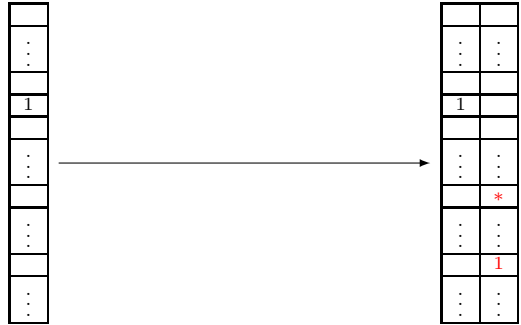

\textbf{Example:} Consider the parabolic defines by the composition $(2,3,1)$ to $(2,3,2)$:
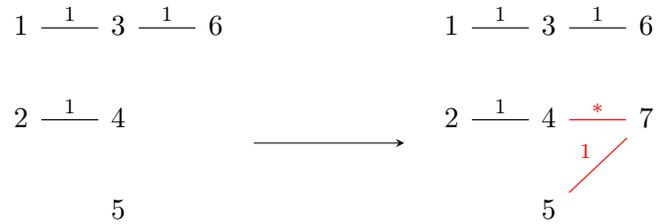
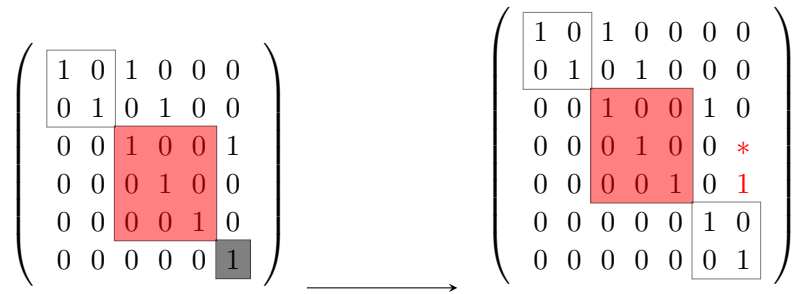
\begin{figure}[hbt!]
 \centering
\begin{tikzcd}[row sep=1.8 em,%tiny,
column sep = 2 em]
1\arrow[-,r,"1"]&3\arrow[-,r,"1"]&6\\
2\arrow[-,r,"1"]&4\\
&5\\
\end{tikzcd}
\;\begin{tikzpicture}

\draw [->,>=stealth] (0,.5) -- (2,.5);
\end{tikzpicture} \;
\begin{tikzcd}[row sep=1.8 em,%tiny,
column sep = 2 em]
1\arrow[-,r,"1"]&3\arrow[-,r,"1"]&6\\
2\arrow[-,r,"1"]&4\arrow[-,r,"\ast",red]&7\\
&5\red{\arrow[-,ur,"1",red]}\\
\end{tikzcd}
\end{figure}

\begin{center}
 \begin{tikzpicture}
 \matrix [matrix of math nodes,left delimiter=(,right delimiter=)] (m)
{
1 & 0 & 1& 0& 0&0 \\
0 & 1 & 0& 1 &0&0 \\
0 & 0 & 1& 0 &0&1 \\
0 & 0 & 0& 1 &0&0\\
0 & 0 & 0& 0 &1&0\\
0 & 0 & 0& 0 &0&1\\
};

\draw[
opacity=0.5] (m-2-1.south west) rectangle (m-1-2.north east);
\draw[
fill=red, opacity=0.5] (m-5-3.south west) rectangle (m-3-5.north east);
\draw[
fill=, opacity=0.5] (m-6-6.south west) rectangle (m-6-6.north east);
\end{tikzpicture}
\;\begin{tikzpicture}

\draw [->,>=stealth] (0,.5) -- (2,.5);
\end{tikzpicture} \;
 \begin{tikzpicture}
 \matrix [matrix of math nodes,left delimiter=(,right delimiter=)] (m)
{
1 & 0 & 1& 0& 0&0&0 \\
0 & 1 & 0& 1 &0&0 &0\\
0 & 0 & 1& 0 &0&1 &0\\
0 & 0 & 0& 1 &0&0&\red{\ast}\\
0 & 0 & 0& 0 &1&0&\red{1}\\
0 & 0 & 0& 0 &0&1&0\\
0 & 0 & 0& 0 &0&0&1\\
};

\draw[
opacity=0.5] (m-2-1.south west) rectangle (m-1-2.north east);
\draw[
fill=red, opacity=0.5] (m-5-3.south west) rectangle (m-3-5.north east);
\draw[
 opacity=0.5] (m-7-6.south west) rectangle (m-6-7.north east);
\end{tikzpicture}

\end{center}

\

\subsection {Case $3$}\label{6.3}

\

\
\begin{figure}[H]
\centering \tiny
\begin{tikzpicture}[grow=right,
sloped,
edge from parent/.style={draw,-latex}, level distance=6cm, sibling distance=5cm]

\node{ $\left.\begin{array}{|c|}\hline \\\hline \vdots \\\hline \\\hline \vdots \\\hline 1 \\\hline \vdots \\\hline \ast \\\hline \vdots \\\hline \\\hline \vdots \\\hline \end{array}\right.$
}
child {
 node{ $\left.\begin{array}{|c|c|}\hline & \\\hline \vdots& \vdots \\\hline & \red{1} \\\hline\vdots & \vdots\\\hline & \red{\ast} \\\hline \vdots&\vdots \\\hline \red{1} & \\\hline \vdots & \vdots \\\hline & \\\hline \vdots&\vdots \\\hline \end{array}\right.$}
edge from parent
node[above] { }
node[below] { }};
\end{tikzpicture}\\
\caption{$S$ has a column of height $s$ and two columns of height $s+1$. This illustrates $(ii),(iii)A$ and $(iv)$.}
%\end{center}
\end{figure}
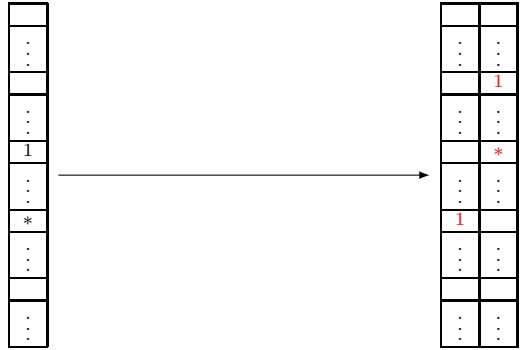

\textbf{Example:} Consider the parabolic defines by the composition $(2,3,2,1,1)$ to $(2,3,2,1,2)$:
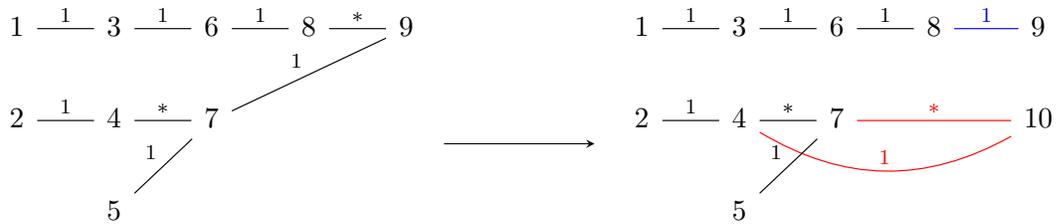
\begin{figure}[hbt!]
\centering
\begin{tikzcd}[row sep=1.8 em,%tiny,
column sep = 2 em]
1\arrow[-,r,"1"]&3\arrow[-,r,"1"]&6\arrow[-,r,"1"]&8\arrow[-,r,"\ast"]&9\\
2\arrow[-,r,"1"]&4\arrow[-,r,"\ast"]&7\arrow[-,urr,"1"]\\
&5\red{\arrow[-,ur,"1"]}\\
\end{tikzcd}
\;\begin{tikzpicture}

\draw [->,>=stealth] (0,.5) -- (2,.5);
\end{tikzpicture} \;
\begin{tikzcd}[row sep=1.8 em,%tiny,
column sep = 2 em]
1\arrow[-,r,"1"]&3\arrow[-,r,"1"]&6\arrow[-,r,"1"]&8\arrow[-,r,"1",blue]&9\\
2\arrow[-,r,"1"]&4\arrow[-,r,"\ast"]\arrow[-,rrr,"1",red, bend right=30]&7\arrow[-,rr,"\ast",red]& &10\\
&5\red{\arrow[-,ur,"1"]}\\
\end{tikzcd}
\end{figure}
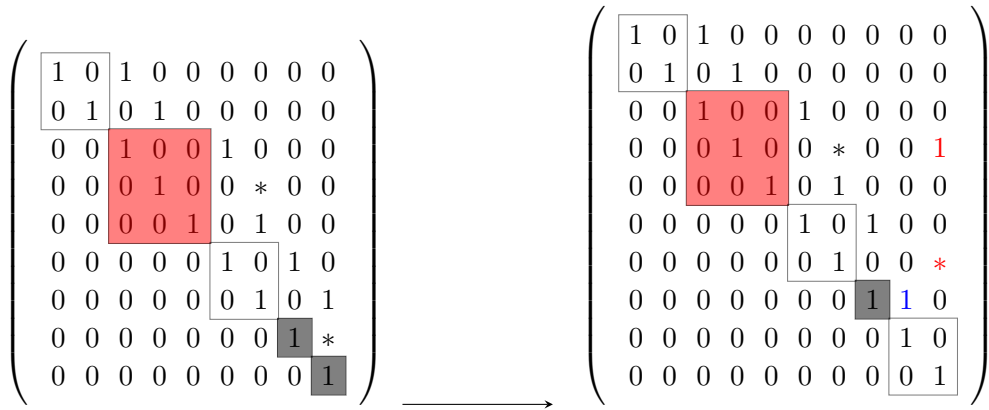
\begin{center}
\begin{figure}[H]
 \begin{tikzpicture}
 \matrix [matrix of math nodes,left delimiter=(,right delimiter=)] (m)
{
1 & 0 & 1 & 0 & 0 & 0 & 0 & 0 & 0 \\
0 & 1 & 0 & 1 & 0 & 0 & 0 & 0 & 0 \\
0 & 0 & 1 & 0 & 0 & 1 & 0 & 0 & 0 \\
0 & 0 & 0 & 1 & 0 & 0 & \ast & 0 & 0 \\
0 & 0 & 0 & 0 & 1 & 0 & 1 & 0 & 0 \\
0 & 0 & 0 & 0 & 0 & 1 & 0 & 1 & 0 \\
0 & 0 & 0 & 0 & 0 & 0 & 1 & 0 & 1 \\
0 & 0 & 0 & 0 & 0 & 0 & 0 & 1 & \ast \\
0 & 0 & 0 & 0 & 0 & 0 & 0 & 0 & 1 \\
};

\draw[
opacity=0.5] (m-2-1.south west) rectangle (m-1-2.north east);
\draw[
fill=red, opacity=0.5] (m-5-3.south west) rectangle (m-3-5.north east);
\draw[
 opacity=0.5] (m-7-6.south west) rectangle (m-6-7.north east);
 \draw[
 fill=, opacity=0.5] (m-8-8.south west) rectangle (m-8-8.north east);
 \draw[
 fill=, opacity=0.5] (m-9-9.south west) rectangle (m-9-9.north east);
\end{tikzpicture}
\;\begin{tikzpicture}

\draw [->,>=stealth] (0,.5) -- (2,.5);
\end{tikzpicture} \;
 \begin{tikzpicture}
 \matrix [matrix of math nodes,left delimiter=(,right delimiter=)] (m)
{
1 & 0 & 1 & 0 & 0 & 0 & 0 & 0 & 0 & 0 \\
0 & 1 & 0 & 1 & 0 & 0 & 0 & 0 & 0 & 0 \\
0 & 0 & 1 & 0 & 0 & 1 & 0 & 0 & 0 & 0 \\
0 & 0 & 0 & 1 & 0 & 0 & \ast & 0 & 0 & \red{1} \\
0 & 0 & 0 & 0 & 1 & 0 & 1 & 0 & 0 & 0 \\
0 & 0 & 0 & 0 & 0 & 1 & 0 & 1 & 0 & 0 \\
0 & 0 & 0 & 0 & 0 & 0 & 1 & 0 & 0 & \red{\ast} \\
0 & 0 & 0 & 0 & 0 & 0 & 0 & 1 & \blue{1} & 0 \\
0 & 0 & 0 & 0 & 0 & 0 & 0 & 0 & 1 & 0 \\
0 & 0 & 0 & 0 & 0 & 0 & 0 & 0 & 0 & 1 \\
};

\draw[
opacity=0.5] (m-2-1.south west) rectangle (m-1-2.north east);
\draw[
fill=red, opacity=0.5] (m-5-3.south west) rectangle (m-3-5.north east);
\draw[
 opacity=0.5] (m-7-6.south west) rectangle (m-6-7.north east);
 \draw[
 fill=, opacity=0.5] (m-8-8.south west) rectangle (m-8-8.north east);
 \draw[
 opacity=0.5] (m-10-9.south west) rectangle (m-9-10.north east);
\end{tikzpicture}
\end{figure}
\end{center}

The three following figures illustrate the VS quadruplet of the  examples of \ref {4.4.8}. \\
\textbf{Examples:}  $(i)$. Consider the parabolic defines by the composition $(1,2,1,1,2,2,3)$:
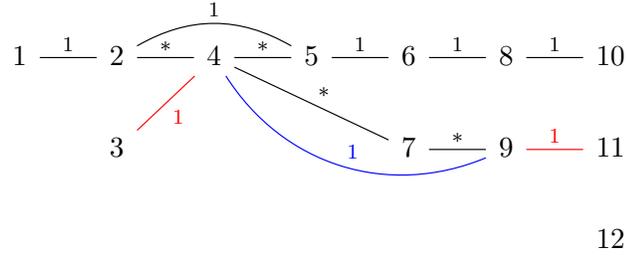
\begin{figure}[H]
\centering
\begin{tikzcd}[row sep=1.8 em,%tiny,
column sep = 2 em]
1\arrow[-,r,"1"]&2\arrow[-,r,"*"]\arrow[-,rr,"1", bend left=30]&4\arrow[-,r,"*"]\arrow[-,ld,red,"1"]\arrow[-,rrrd,"1", bend right=40,blue]\arrow[-,rrd,"*"]&5\arrow[-,r,"1"]&6\arrow[-,r,"1"]&8\arrow[-,r,"1"]&10\\
 &3& & &7\arrow[-,r,"*"]&9\arrow[-,r,red,"1"]&11\\
 && & &&&12\\
\end{tikzcd}
\caption{The VS pair $(3, 4, 9, 11)$ in red is not bad thanks to the blue line $\ell_{4,9}$. }
\end{figure}

 $(ii)$.  Let the parabolic defines by the composition $(1,2,1,1,2,3)$:
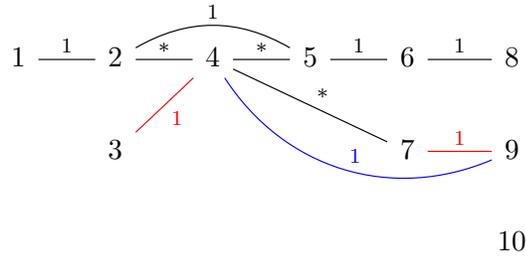
\begin{figure}[H]
\centering
\begin{tikzcd}[row sep=1.8 em,%tiny,
column sep = 2 em]
1\arrow[-,r,"1"]&2\arrow[-,r,"*"]\arrow[-,rr,"1", bend left=30]&4\arrow[-,r,"*"]\arrow[-,ld,red,"1"]\arrow[-,rrrd,"1", bend right=40,blue]\arrow[-,rrd,"*"]&5\arrow[-,r,"1"]&6\arrow[-,r,"1"]&8\\
 &3& & &7\arrow[-,r,"1", red]&9\\
 && & &&10\\
\end{tikzcd}
\caption{the VS pair $(3, 4, 7, 9)$ in red is not bad. the line $\ell_{4,9}$ representing $x_{4,9}\in \supp e $ lies in the same column as $x_{4,9}$. }
\end{figure}

\

 $(iii)$.  Let the parabolic defines by the composition $(3,2,1,1,2,3)$:\\

\begin{figure}[H]
\centering
\begin{tikzcd}[row sep=1.8 em,%tiny,
column sep = 2 em]
1\arrow[-,r,"1"]&4\arrow[-,r,"1",red]&6\arrow[-,r,"*"]\arrow[-,rrd,"*", bend right=0]\arrow[-,rrrd,"1", bend right=30,green]\arrow[-,rrrdd,"*", bend right=40]&7\arrow[-,r,"1"]&8\arrow[-,r,"1"]&10\\
2\arrow[-,r,"1"]&5\arrow[-,rru,"1"]& & &9\arrow[-,r,"1",red]&11\\
3\arrow[-,rrrru,"1"]&&&&&12\\
\end{tikzcd}
\caption{The VS pair $(3, 4, 7, 9)$ in red is bad, we augment the $e$ with the green line $\ell_{4,10}$. }\label{fig3}
\end{figure}
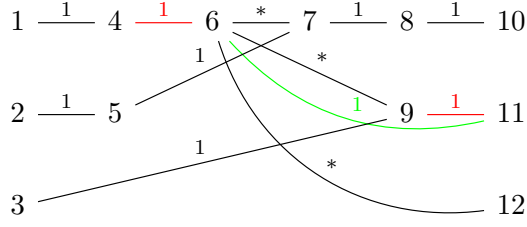

Recall \cite [4.2]{FJ2} that $e$ is given by the co-ordinate vectors defined by the lines with label $1$.
In figure \ref{fig3}, these form three disjoint lines $(1,4,6)$, $(2,5,7,8,10)$, $(3,9,11)$, so $e$ has nilpotency class is $(5,3,3,1)$ giving rise to column heights $(4,3,3,1,1)$. Thus $ \frac{1}{2}\dim \vert G\cdot e =\dim \mathfrak{n} -\frac{1}{2}(4\times3 + 3\times2 + 3\times2)=\dim \mathfrak{n} -12$, whilst $\dim \mathfrak{m}= \dim \mathfrak{n}-8$. Consequently $\frac{1}{2} \dim G\cdot e = \dim P\cdot e = \dim \mathfrak{m}- 4$. Yet $e\in \mathscr N$ is regular if and only if $\dim P\cdot e = \dim \mathfrak{m}- \dim V =\dim \mathfrak{m}- 3$.\\

The lines in red represent the bad VS pair $(3, 4, 7, 9)$ in figure \ref{fig3}.
By \ref {4.4.4} we must adjoin the green line, that is $\ell_{6,11}$, giving $e_{VS}=e+x_{6,11}$, which by Prop. \ref {4.5.3} lies in $\mathscr N^e$. The nilpotency class of $e_{VS}$ is $(5,4,2,1)$ giving rise to column heights $(4,3,2,2,1)$. Thus $ \frac{1}{2} \dim G\cdot e_{VS} =\dim \mathfrak{n}-11$ and $ \frac{1}{2} \dim G\cdot e_{VS} = \dim P\cdot e_{VS} = \dim \mathfrak{m} - 3=\dim \mathscr N^e$, the last equality by Prop. \ref {4.5.3}.  Thus $Pe_{VS} $ is dense orbit in $\mathscr N^e$ and so $e_{VS}$ is a regular element in $\mathscr N$.

 \section {Index of Notation }\label {7}

 Symbols used frequently are given below in the order in which they appear.

 \

 \ref {1}. \quad \ $\mathbb C, [1,n]$.

 \ref {1.1}. \ $G,\mathfrak g, \mathfrak h, H, \Delta, \Delta ^+, \pi, \mathfrak n, \mathfrak n^-, \mathfrak b, W, s _\alpha, \alpha^\vee$.

 \ref {1.2}. \ $\mathfrak r_{\pi'}, \mathfrak p_{\pi'},\mathfrak m_{\pi'}, P, P', \mathfrak p'$.

 \ref {1.3}. \ $\mathscr N, e+V, \mathscr N^e$.

 \ref {1.4}. \ $\mathscr N_{reg}, e_{VS}, E_{VS}$.

 \ref {1.6}. \ $B.\mathfrak u, d$.

 \ref {2.1}. \ $\textbf{M}_n, x_{i,j}, \alpha_{i,j}$.

 \ref {2.2}. \ $W_{\pi'},w_{\pi'},c_i, B_i, \textbf{C}_i$.

 \ref {2.3}. \ $\mathscr D_\mathfrak m, R_i,R^i, C_i, \mathscr T_\mathfrak m, \ell_{b,b'}, \ell_{i,j}$.

 \ref {2.4} \ $M_{v,v'}, \mathfrak u_{v,v'}, w_{u,u'}$.

 %\ref {2.4}. \ $c_i^{\leq s},c_i^{>s},\mathscr F_\mathfrak m (n), \mathscr D_\mathfrak m (n),t_m(n)$.

% \ref {2.5}. \ $w_c(\mathscr T),S(w)$.

 \ref {2.6}. \ $M_s(v,v'), \mathscr T^t, M_s, M_s(\mathfrak m)$.

 \ref {2.7}. \ $\mathscr P$.

 \ref {3.2.1}. $\mathscr C, \ell(\mathscr C)$.

 \ref {3.3}. \ $b_{i,j}$.

 \ref {3.3.1}. $\textbf{C}(s), \textbf{M}(s),s, \textbf{S}, \textbf{M}^k,e^\kappa,V^\kappa$.

 \ref {3.3.2}. $\textbf{r}_{C,C'}$.

 \ref {3.4}. \ $\textbf{C}^t$.

 \ref {3.4.2}. $\mathscr R, |S|$.

 \ref {3.6.8}. $\mathscr S$.

 \ref {4.4.2}. $C<C', b(t), C_{(t)},c_{(t)}$.

 \ref {4.4.4}. $e^\kappa_{VS},E^\kappa_{VS}$.

 \

\textbf{Acknowledgements.}

\

This work was the subject of a joint talk on Zoom hosted by the Weizmann Institute held during 25 April to 6 May 2021, on the occasion of the $60^{th}$ birthday of Vera Serganova. We would like to thank the organizers (Inna Entova-Aizenbud, Maria Gorelik and Dimitar Grantcharov) for the invitation to speak. Our talk may be accessed on

https://innaentova.wixsite.com/springfest2021/recording-links,
using the password $362880$, which is $9!$.

The first author was supported at the Weizmann Institute, arranged by Ronen Basri and Gal Binyamini on ISRAEL SCIENCE FOUNDATION (grant No. 1167/17) and the MINERVA Stiftung with the funds from the BMBF of the Federal Republic of Germany.

During the academic year $2021-2022$ the first author was supported through ISF grant 1017/19 as a postdoctoral student of Anna Melnikov at the University of Haifa.

\

\end{document}